\documentclass[11pt]{article}
\usepackage{paper}

\usepackage{amsfonts}
\usepackage{amssymb}
\usepackage{amsthm}
\usepackage{hyperref}
\usepackage{comment}
\usepackage[capitalise]{cleveref}
\crefname{subsection}{Subsection}{Subsections}

\newcommand{\papertitle}{Adaptive Consensus: A network pruning approach for decentralized optimization}
\newcommand{\paperauthora}{Albert S. Berahas}
\newcommand{\paperauthoraaffiliation}{Department of Industrial and Operations Engineering, University of Michigan}
\newcommand{\paperauthoraemail}{albertberahas@gmail.com}
\newcommand{\paperauthorb}{Raghu Bollapragada}

\newcommand{\paperauthorbemail}{raghu.bollapragada@utexas.edu}
\newcommand{\paperauthorc}{Suhail M. Shah}
\newcommand{\paperauthorcaffiliation}{Operations Research and Industrial Engineering Program, University of Texas at Austin}
\newcommand{\paperauthorcemail}{suhailshah2005@gmail.com}

\usepackage{booktabs}
\usepackage{caption}
\usepackage{float}
\usepackage{titlesec}
\usepackage{capt-of}

\usepackage{array}
\usepackage{arydshln}
\begin{document}
\title{\papertitle}

\author{\paperauthorc\footnotemark[1]
   \and \paperauthora\footnotemark[2]
   \and \paperauthorb\footnotemark[1]}

\maketitle

\renewcommand{\thefootnote}{\fnsymbol{footnote}}

\footnotetext[1]{\paperauthorcaffiliation. (\url{\paperauthorcemail},\url{\paperauthorbemail})}
\footnotetext[2]{\paperauthoraaffiliation. (\url{\paperauthoraemail})}
\renewcommand{\thefootnote}{\arabic{footnote}}

\begin{abstract}{
We consider network-based decentralized optimization problems, where each node in the network possesses a local function and the objective is to collectively attain a consensus solution that minimizes the sum of all the local functions. A major challenge in decentralized optimization is the reliance on communication which remains a considerable bottleneck in many applications. To address this challenge, we propose an adaptive randomized communication-efficient algorithmic framework that reduces the volume of communication by periodically 
tracking 
the disagreement error and judiciously selecting the most influential and effective edges at each node for communication. Within this framework, we present two algorithms: Adaptive Consensus (\texttt{AC}) to solve the consensus problem and Adaptive Consensus based Gradient Tracking (\texttt{AC-GT}) to solve smooth strongly convex decentralized optimization problems. We establish strong theoretical convergence guarantees for the proposed algorithms and quantify their performance in terms of various algorithmic parameters under standard assumptions. Finally, numerical experiments showcase the effectiveness of the framework in significantly reducing the information exchange required to achieve a consensus solution.

}
\end{abstract}


\section{Introduction}
The problem of network-based decentralized optimization can be formally stated as,
\begin{equation}\label{mainprob}
\begin{aligned}
    \min_{x_i \in \mathbb{R}^d} \quad & \frac{1}{n} \sum_{i=1}^n f_i(x_i)\\
    \text{s.t.} \quad & x_i =x_j,\,\forall\, i,j \in [n]:=\{1,2,\cdots,n\},
\end{aligned}
\end{equation}
where $f_i(\cdot) : \mathbb{R}^d \rightarrow \mathbb{R}$ is a component of the objective function located at node $i \in [n]$, 
and $x_i \in \mathbb{R}^d$ is a copy of the optimization variable at node $i \in [n]$. A closely related yet simplified version of this problem, whose goal is to reach consensus among the nodes, i.e., $x_i=x_j$ for all $i\in [n]$, without minimizing an objective function, is referred to as the consensus problem \cite{rensurvey}. 
Problems of these types arise in several applications including wireless sensor networks \cite{pc3,pc4}, power systems design \cite{pc5,decent_pow}, parallel computing \cite{pc1,pc2}, and robotics \cite{pc6,rob2}. More recently, decentralized optimization has experienced renewed interest owing to the abundance of decentralized data and privacy-preserving machine learning \cite{pc7,pc8}, where $f_i$ is a function of the data held by node $i\in[n]$. Several classes of decentralized optimization algorithms have been proposed to solve \eqref{mainprob}, where the main components consist of local computations at every node and information exchange (communication) between nodes in order to achieve consensus~\cite{pc1}. The communication requirement in many applications remains a major bottleneck in the performance of decentralized optimization methods \cite{tsit1,survey,do1,com2,com3,com4}.

In this work, we propose and develop a novel approach to reduce the communication requirements in decentralized optimization without significantly impacting the convergence properties of the underlying algorithm. The core principle of our approach involves judiciously selecting a subset of the edges of the network (instead of all the edges) along which communication is performed at each iteration, thereby reducing the communication efforts. A key observation motivating this approach is that selectively pruning the edges of the network has marginal impact on the spectral properties of the mixing matrix associated with any graph topology. This matrix plays a crucial role in determining the rate of information diffusion through the network~\cite{tsit1}, which subsequently affects the rate of achieving consensus amongst nodes. In fact, for many network structures, the spectral properties remain virtually unchanged even after selectively pruning up to 50-60\% of the edges (see Section 4.1),  thus retaining a 
consensus rate akin to that of an unpruned network while reducing the communication volume. 


However, to fully leverage the potential of such pruning approaches, one requires information about the most influential edges, i.e., the edges that achieve consensus with minimal communication cost, information that is typically unknown. 
For example, 
the bridge edge that connects two fully connected components in a barbell graph \cite[Figure 2]{barbell} has 
a significantly more influential role in the consensus process than other edges. Therefore, it is beneficial to communicate along the bridge edge as compared to other edges. Unfortunately, due to the decentralized nature of the network, nodes cannot a priori 
determine these influential 
edges. Moreover, the relative influence of different edges in achieving consensus can vary significantly depending on the network state and structure, and the application. 
To overcome this challenge, our work proposes a cyclic adaptive randomized procedure that can be implemented in a decentralized manner to identify 
such edges and reduce the communication costs. Specifically, we periodically track the \emph{disagreement error} along edges during the consensus process to estimate the relative importance of edges in achieving consensus and maintain a network with only 
the most influential edges.

\subsection{Contributions}
A concise summary of the contributions  
is as follows:
\begin{itemize}[leftmargin=0.75cm]
    \item We propose an adaptive communication-efficient algorithmic framework. Within this framework, we introduce two new algorithms: Adaptive Consensus (\texttt{AC}) to solve the consensus problem and Adaptive Consensus based Gradient Tracking (\texttt{AC-GT}) to solve the decentralized optimization problem\footnote{For better exposition of the consensus framework, the consensus and decentralized optimization problems are treated separately even though the former is a simplified version of the latter.}. 
    The novelty in our approach lies in the ability to exploit the underlying structure of the network to reduce the volume of communication. This is accomplished via an adaptive consensus scheme that selects the most influential and effective edges for communication at each node based on the graph topology. The proposed framework has broad applicability and can be integrated with other existing decentralized optimization algorithms or adapted to other settings including directed graphs, time-varying topologies, and asynchronous updates.

    \item We provide theoretical convergence guarantees for smooth strongly convex problems for both \texttt{AC} and \texttt{AC-GT}, demonstrating that they retain the linear convergence properties of their base counterparts, i.e., methods that do not utilize the adaptive consensus framework, while requiring reduced communication. 
    The analysis  utilizes the inhomogeneous matrix product theory to prove linear convergence by showing that the pruned matrix products remain contractive.  In contrast to prevalent analytical approaches in decentralized optimization with time-varying graphs, the rate constant in our results is obtained using the coefficient of ergodicity which effectively highlights the dependence of the convergence rate on the network pruning procedure parameters.

    \item We illustrate the empirical performance of \texttt{AC} in solving the standard consensus problem and of  \texttt{AC-GT} in solving linear regression and binary classification logistic regression problems. Our numerical results highlight 
    that the proposed methods achieve significant communication savings while maintaining solution quality, compared to the contemporary state-of-the-art techniques.
\end{itemize}

\subsection{Literature Review}
The proposed idea of exploiting the relative significance of edges to improve algorithmic efficiency 
is not exclusive to 
decentralized optimization 
and has been studied in other fields that use graphical modeling on networks \cite{braess,braess1,lottery1,lottery2}. 
In the context of traffic modeling, a converse analogue falls under the category of ``Braess's paradox'', which suggests that adding one or more roads to a road network can actually slow down the overall traffic flow \cite{braess,braess1}. Another example, although somewhat tangential, is found in neural networks where the ``lottery ticket hypothesis'' states that within dense, feed-forward networks, there are smaller pruned sub-networks that, when trained in isolation, can achieve test accuracy comparable to the original network in a similar number of iterations \cite{lottery1,lottery2}.

Within decentralized optimization, several recent works have proposed communication-efficient algorithms that balance the communication and computation costs to achieve overall efficiency \cite{bcc1,bcc2,bcc3,bcc4,bcc5,albert1,diffusion}. Our proposed approaches are complementary to and can be integrated with these existing works. Furthermore, the proposed framework (adaptive consensus) adds to the list of techniques that reduce the communication costs. One such approach is gossip communication protocols where nodes selectively communicate with neighbors asynchronously~\cite{randomgossip,max1,max2,max3}. It is worth noting that in gossip protocols a convex optimization problem is often solved to optimize the spectral gap of the expected consensus matrix~\cite{randomgossip}. 
Another class of approaches leverage quantized communication where only quantized (reduced size) information is communicated to reduce the communication costs. However, these techniques typically lack convergence guarantees to the solution \cite{pc1, suh_rag}. Moreover, quantization techniques can also be incorporated into our framework to further reduce the communication overhead. 
We emphasize that our approach differs significantly from the aforementioned approaches in several ways including the focus on enhancing communication efficiency by adaptively modifying the graph structure in a decentralized manner, and achieving convergence guarantees to the solution.

While several classes of algorithms have been proposed for solving 
decentralized optimization, gradient tracking methods have emerged as popular alternatives due to their simplicity, optimal theoretical convergence properties and empirical performance \cite{EXTRA,NEXT,DIG,aug_dgm,sgt1,albert1}. 
We incorporate the proposed communication-efficient technique into the gradient tracking algorithmic framework with the goal of reducing the communication costs while retaining optimal convergence guarantees. Furthermore, we note that the setting of time-varying graphs, which also arises in our work, has been explored previously in \cite{ev2, alned, assran, DIG}, among others. 

\subsection{Organization}
The paper is organized as follows. In the remainder of this section, we define the notation employed in the paper. 
In Section~\ref{sec.ada_con}, we describe the network model, introduce the Adaptive Consensus (\texttt{AC}) algorithm, and establish convergence guarantees under standard assumptions. Building upon the adaptive consensus procedure and gradient tracking algorithms, we propose the Adaptive Consensus based Gradient Tracking (\texttt{AC-GT}) algorithm and study its convergence properties in Section~\ref{sec.acgt}. Section~\ref{sec.num_red} presents numerical results that illustrate the performance of the proposed algorithms. Finally, concluding remarks are provided in Section~\ref{sec.conc}.

\subsection{Notation} 
We use $\mathbb{R}$ to denote the set of real numbers and $\mathbb{N}$ to denote the set of all strictly positive integers. The $\ell_2$-inner product between two vectors is denoted by 
$\langle \cdot,\cdot \rangle$ and 
$\otimes$ denotes the Kronecker product between two matrices. All norms, unless otherwise specified, can be assumed to be $\ell_2 $-norms of a vector or matrix depending on the argument. 
Let $\lfloor x \rfloor$ ($\lceil x \rceil$) denote the nearest integer less (greater) than or equal to $x$. 
We use $a|b$ to denote integer division between any two $a,b \in \mathbb{N}$, i.e., $a|b = \lfloor a/b \rfloor$. We use $\textbf{1}_{n}:=\frac{1}{n}1_n\otimes I_d \in \mathbb{R}^{nd\times d}$, where $1_n\in \mathbb{R}^n$ is the column vector of all ones and $I_d \in \mathbb{R}^{d\times d}$ is the $d\times d$ identity matrix. For any matrix $Q$ with eigenvalues $-1< \lambda_n\leq \cdots \leq\lambda_2<\lambda_1=1$, the \textit{spectral gap} 
is defined as $
\sigma(Q):= 1 -  \max\{|\lambda_n|,|\lambda_2|\}.$ 
The  set $A\setminus B$  consists of the elements of $A$ which are not elements of $B$. We use $x^*$ denotes the optimal solution of \eqref{mainprob}. 
We use the column vector $x_{i,k} \in \mathbb{R}^d$ to denote the value of the objective variable held by node $i$ at iteration $k$. The vector $\x_k \in \mathbb{R}^{nd}$ denotes the column-stacked version of $x_{i,k}$ and  $\n \textbf{f}(\x_k)$ denotes the column-stacked gradients, i.e., 
\begin{equation*}
    \x_k:= [x_{1,k},\cdots, x_{n,k}] \in \mathbb{R}^{nd} \quad \text{and} \quad \n \textbf{f}(\x_k) := \big[\n f_1 (x_{1,k}), \cdots\, \n f_n(x_{n,k})] \in \mathbb{R}^{nd} ,
\end{equation*}
where $\n f_i:\mathbb{R}^d \to \mathbb{R}^d$ is the gradient of the local function $f_i$. The following quantities are used in the presentation and analysis of the algorithms,  
\begin{equation*}
    \bar{x}_k := \frac{1}{n} \sum_{i=1}^n x_{i,k}
    \in \mathbb{R}^d,\,\,\, \bar{\x}_k = [\bar{x}_k,\cdots,\bar{x}_k]\in \mathbb{R}^{nd},\,\,\,\nabla f(\bar{x}_k):= \frac{1}{n}\sum_{i=1}^n \nabla f_i(\bar{x}_k)\in \mathbb{R}^d.
\end{equation*}

\section{Adaptive Consensus}\label{sec.ada_con}
This section provides a 
description of the pruning protocol which serves as the basic building block for the proposed consensus scheme referred to as the Adaptive Consensus algorithm (Algorithm~\ref{alg.ada_cons}, \texttt{ADAPTIVE CONSENSUS (AC)}). We 
describe the network model we assume in the paper, 
discuss the pruning protocol, and 
present the algorithm and its associated convergence guarantees.

\subsection{Network Model}

The underlying network is assumed to be modeled by a undirected graph $\mathcal{G} =\{\mathcal{V},\mathcal{E}\}$, where $\mathcal{V}$ is the set of nodes and $\mathcal{E}$ is the set of edges. 
We use the matrix $Q = [q_{ij}]_{i\in[n],j\in[n]}$ to denote the mixing matrix. The mixing matrix has the following properties: 
the entry $q_{ij}>0$ (assumed to be equal to $q_{ji}$) if there is a link between any two nodes $i,\,j\in \mathcal{V}$. We use $\E_i$ to denote the set of all edges $(i,j)$ such that $j \in \mathcal{V}$ is a neighbor of $i \in \mathcal{V}$, i.e., the set of all $j\in \mathcal{V}$ with $j\neq i$ for which $q_{ij}>0$. Note that the neighbors of $i$ for any $i \in [n]$ is the set of all $j$ such that $(i,j) \in \E_i$. Since we assume that the graph is undirected, 
$(i,j) \in \E_i$ if and only if 
$(j,i) \in \E_j$. We make the following assumption on the network.

\begin{assumption}[Graph Connectivity]\label{asmp1} $\mathcal{G} =\{\mathcal{V},\mathcal{E}\}$ is static and connected.\end{assumption}

\subsection{Pruning Protocol}

The main goal of the pruning protocol is to provide a systematic approach for selecting 
the 
(subset of) edges within a graph along which to communicate in order to achieve consensus with reduced communication 
efforts. To be more precise, given the reference graph $\G(\V,\E)$ and a set of node estimates $a_i$ for all $i \in [n]$, the pruning protocol generates a modified graph $\G(\V,\bar{\E})$ by selectively removing edges from the reference graph. The edges to be pruned are determined by a function of the node estimates. The function assigns a probability to each edge in $\E$ based on its likelihood of being least effective and influential with respect to 
achieving consensus. 
The pseudo-code for the pruning protocol is given in 
Algorithm~\ref{alg.prun_prot}. 

\begin{algorithm}[]
\footnotesize
   \caption{\small  \texttt{PRUNING PROTOCOL}($\mathcal{G}(\mathcal{V},\mathcal{E}),a_i, 
   (\bar{\kappa}_i,\ubar{\kappa}_i),\beta)$.\label{alg.prun_prot}}
    \footnotesize
     \textbf{Inputs:} Graph $\mathcal{G}(\mathcal{V},\mathcal{E})$; Node estimates $a_i$ for all $i \in [n]$; 
     Softmax parameter $\beta \in [0,\infty]$; Thresholding factors $(\bar{\kappa}_i,\ubar{\kappa}_i) \in [0, 1]^2$ for all $i \in [n]$.
  \begin{algorithmic}[1]
   \footnotesize
          \STATE Set $\E^{\text{prune}}_i := \{\}$ for all $i \in [n]$.
       \FORALL{$i\in [n]$ in parallel}
       \STATE Receive estimates $a_j$ from all neighbors $j$.
       \STATE Compute a dissimilarity measure 
       $\Delta(a_i,a_j)$ for all edges $(i,j) \in \E_i$.
        \WHILE{$ |\E^{\text{prune}}_i| \leq \lfloor\bar{\kappa}_i \times |\mathcal{E}_i|\rfloor$}
            \STATE Draw a sample edge $(i,j')$ from $\E_i \setminus \E^{\text{prune}}_i $ according to:
            $$
            p_{i,j} \sim \tfrac{\exp(-\beta \Delta(a_i,a_j))}{\sum_{ (i,j') \in \E_i \setminus \E^{\text{prune}}_i }\exp(-\beta \Delta(a_i,a_{j'}))}, \qquad ((i,j) \in \E_i  \setminus \E^{\text{prune}}_i).
            $$
            \STATE Update set $\E^{\text{prune}}_i \to \E^{\text{prune}}_i \cup (i,j')$ for all $i \in [n]$.
       \ENDWHILE
       \ENDFOR \vspace{0.1cm}
       \STATE Set $ \bar{\E}_i:= \E_i,$ 
       for all $i \in [n]$.
       \FORALL{all $i\in [n]$}
            \STATE Send requests to all neighbors $j$ such that $(i,j) \in \E_i^{\text{prune}}$ to prune edge $(j,i) \in \E_j$.
            \STATE Receive request from all {neighbors} $j'$ such that $(j',i) \in \E_{j'}^{\text{prune}}$ to prune edge $(i,j') \in \E_i$. 
            \FORALL {$(i,j')$ such that $(i,j') \in \E^{\text{prune}}_{i} $}
            \STATE Remove edge
 $(i,j')$ from $\bar{\E}_i$.
             \ENDFOR
             \FORALL {requests  $(i,j')$ such that $(i,j') \notin \E^{\text{prune}}_{i} $}
             \IF { $|\bar{\E}_{i}|>\lceil  \ubar{\kappa}_i |\mathcal{E}_{i}|\rceil $} 
             \STATE Remove edge $(i,j')$ from $\bar{\E}_i$.
             \ENDIF
            \ENDFOR
       \ENDFOR \vspace{0.1cm}
\IF{\textit{Graph}$=$\textit{`Undirected'}}
 \FORALL{ 
 $(i,j)\in \bar{\E}_i$ and   $(j,i)\notin \bar{\E}_j$}
 \STATE Update set $\bar{\E}_j \to \bar{\E}_j \cup (j,i)$. 
 \ENDFOR
 \ENDIF
  \end{algorithmic}
  \textbf{Output:} $\G(\V,\bar{\EB}) $, where $\bar{\EB}:= \cup_{i=1}^n \bar{\E}_i$.
\end{algorithm}

Algorithm~\ref{alg.prun_prot} has three free (user-defined) parameters ($\bar{\kappa}_i,\,\ubar{\kappa}_i $ and $\beta$). 
Broadly speaking, $\bar{\kappa}_i \in [0,1]$ represents the fraction of edges to be pruned 
at node $i \in [n]$ and $\ubar{\kappa}_i \in [0,1]$ is a lower bound on the minimum number of edges retained at node $i$. The parameter $\beta \in [0,\infty]$ determines the level of influence of the dissimilarity measure in assigning the pruning probabilities. The role and significance of these parameters becomes evident by 
examining the main steps of the protocol, which we discuss next.

\paragraph{Selecting Candidate Edges for Pruning}

To select 
the edges to be pruned, each node $i \in [n]$ constructs a set $\E_i^{\text{prune}}$ by iteratively drawing a sample edge from the set $\E_i \setminus \E^{\text{prune}}_i$, $\lfloor \bar{\kappa}_i\times  |\E_i| \rfloor$ times, where 
$\bar{\kappa}_i$ represents the fraction of the total number of edges to be removed at node $i$ during pruning. 
The probability of selecting an edge $(i,j)$ is determined by the softmax of a dissimilarity measure (denoted by $\Delta(a_i,a_j)$) between  the estimates at $i$ and $j$. A possible candidate for $\Delta(a_i,a_j)$ is the $\ell_1$-norm difference between $a_i$ and $a_j$, i.e., $\|a_i-a_j\|_1$. For large values of the parameter $\beta$ (the argument of the softmax) edges exhibiting small dissimilarity (small $\Delta(a_i,a_j)$), where $a_i$ and $a_j$ are in similar, have an increased likelihood of being pruned.

More formally, for the $k$th draw at node $i \in [n]$, where $1\leq k\leq \lfloor \bar{\kappa}_i  |\E_i| \rfloor$, the probability distribution over the set of edges $(i,j) \in \E_i \setminus \E^{\text{prune}}_i$ is given by 
\begin{align*}
            p_{i,j} \sim \tfrac{\exp(-\beta \Delta(a_i,a_j))}{\sum_{ (i,j') \in \E_i/\E^{\text{prune}}_i } \exp(-\beta \Delta(a_i,a_{j'}))}, \qquad \text{for all} \;\; (i,j) \in \E_i \setminus \E^{\text{prune}}_i,    
\end{align*}
where $\beta \in [0,\infty]$ is the softmax parameter that controls the influence of the dissimilarity measure. Note that $\beta =\infty$ represents the greedy case, where each node $i \in [n]$ selects the top $\lfloor \bar{\kappa}_i  |\E_i| \rfloor$ edges with least dissimilarity measure. At the other extreme, $\beta=0$ represents the case of random pruning independent of the dissimilarity measure. 

\paragraph{Pruning Mechanism}
To perform the actual pruning, each node $i \in [n]$ sends a request to neighboring nodes $j$, where $(i,j) \in \E^{\text{prune}}_i$, to prune edge $(j,i)$. 
At the same time, node $i \in [n]$ receives and catalogues the requests from all its neighboring nodes $j'$ with $(j',i ) \in \E_{j'}^{\text{prune}}$ to prune edges $(i,j')$. 
It is worth noting that the request for $(i,j')$ does not necessarily require $(i,j')$ to be in $\E_i^{\text{prune}}$. Initially, each node creates a copy $\bar{\E}_i$ of the original set of edges 
$\E_i$. The following steps are then performed in order by each node:
\begin{enumerate}[leftmargin=0.75cm]
    \item[\emph{(i)}] For each $(i,j') $ such that $(i,j') \in \E_i^{\text{prune}}$, edge $(i,j')$ is removed from $\bar{\E}_i$. This covers the ideal case where both nodes $i$ and $j'$ want to remove the edge $(i,j')$ and $(j',i)$ from their respective edge sets $\E_i$ and $\E_{j'}$.
    \item[\emph{(ii)}] If $(i,j') \notin \E^{\text{prune}}_i$, then the edge is pruned if $|\bar{\E}_{i}|> \lceil  \ubar{\kappa}_i |\mathcal{E}_{i}|\rceil $. So, node $i \in [n]$ prunes an edge not included in $\E_i^{\text{prune}}$ only if the number of edges remaining in $\bar{\E}_i$ \textit{is greater than a certain fraction} $\ubar{\kappa}_i$ of $|\mathcal{E}_{i}|$. An implicit assumption here is that $\ubar{\kappa}_i \leq 1-\bar{\kappa}_i$ so that $\lceil  \ubar{\kappa}_i |\mathcal{E}_{i}|\rceil  \leq \lceil  (1-\bar{\kappa}_i) |\mathcal{E}_{i}|\rceil  $. It should be noted that for the algorithm to be well-defined, pruning requests of this type are processed in the order in which they are received. 
\end{enumerate}
The output of Algorithm~\ref{alg.prun_prot} is $\G(\V,\bar{\EB})$, where $\bar{\EB}:=\cup_i \bar{\E}_i$. An important point worth noting here is that the resulting set $\bar{\E}_i$ for $i \in [n]$ may contain edges $(i,j)$ for which $(j,i)\notin \bar{\E}_j$. To make the pruned graph undirected, there are two possible approaches; either node $j$ adds $(j,i)$ to $\bar{\E}_j$, or alternatively, node $i$ removes $(i,j)$ from $\bar{\E}_i$. These approaches can be implemented by performing one additional round of communication among the nodes with negligible overhead.

\subsection{Adaptive Consensus}
Building upon the pruning protocol presented in the previous subsection, we introduce an algorithm to solve the consensus problem \cite[Section 1]{alex_tsit}, which requires the convergence of all the node estimates to the average of their initial estimates. The pseudo-code is provided in Algorithm~\ref{alg.ada_cons}.

\begin{algorithm}[]
\footnotesize
   \caption{\small \texttt{ADAPTIVE CONSENSUS (AC)}\label{alg.ada_cons}}
    \footnotesize
    \textbf{Inputs:} Graph $\mathcal{G}(\mathcal{V},\mathcal{E})$; Cycle length $\tau \in \mathbb{N}$; Softmax parameter $\beta \in [0,\infty]$; Thresholding factors $(\bar{\kappa}_i,\ubar{\kappa}_i) \in [0, 1]^2$ for all $i \in [n]$; 
    Initial estimates $x_{i,0} \in \mathbb{R}^d$ for all $i\in [n]$; Total number of iterations $T\in \mathbb{N}$.
  \begin{algorithmic}[1]
   \footnotesize
    \FOR{$k=0,\dots,T$}
     \FOR{ all $i\in [n]$ in parallel} 
          \IF{$k \in \mathcal{I},$}
          \STATE Generate $\mathcal{G}(\mathcal{V},\bar{\EB}_{k|\tau}) \sim$ \texttt{PRUNING PROTOCOL}($\mathcal{G}(\mathcal{V},\mathcal{E}),x_{i,k},(\bar{\kappa}_i,\ubar{\kappa}_i),\beta)$.
           \STATE Get new weights $\bar{q}_{ij}[k|\tau] \sim $ \texttt{GENERATE WEIGHTS} ($\mathcal{G}(\mathcal{V},\bar{\EB}_{k|\tau})$).
          \ENDIF
         \STATE Update estimate at node $i$ according to:
          $
          x_{i,k+1} = \sum_{j=1}^n \bar{q}_{ij}[k|\tau]x_{j,k}.
          $
       \ENDFOR  
    \ENDFOR
  \end{algorithmic}
  \textbf{Output:} $x_{i,T}$ for all $i \in [n]$.
\end{algorithm}
We discuss the main steps of the algorithm and how to select the parameters $\bar{\kappa}_i$ and $\ubar{\kappa}_i$. Algorithm~\ref{alg.ada_cons} has a cyclic structure with cycle length 
$\tau\in \mathbb{N}$. The set of indices where the pruning protocol is executed is denoted by $\mathcal{I}$. For any $k \in \mathcal{I}$, the iterations $t\in [k,k+\tau)$ are said to constitute a \textit{consensus cycle}.

\paragraph{Pruning Step} 
At the start of the $k|\tau$ consensus cycle, the pruning protocol is executed to obtain the pruned graph $\G(\V,\bar{\EB}_{k|\tau})$, where $\bar{\EB}_{k|\tau}:= \cup_i \bar{\E}_{i,k|\tau}$, using the current local estimates 
$x_{i,k}$ for all $i\in [n]$. Subsequently, the mixing matrix, denoted by $Q_{k|\tau}  := [q_{ij}[k|\tau]]_{i\in [n],j\in [n]} $, of the pruned graph $\G(\V,\bar{\EB}_{k|\tau})$ is constructed in a decentralized manner. As an example, 
we can consider the Metropolis-Hastings scheme \cite{DIG}, which generates the weights via the following prescribed rule: 
\begin{equation}\label{MH}
  q_{ij}[k|\tau] :=
    \begin{cases}
      \tfrac{1}{\left(1+\max\{ |\bar{\E}_{i,{k|\tau}}|,|\bar{\E}_{{j,k|\tau}}|\}\right)} &  \text{if } (i,j) \in \bar{\EB}_{k|\tau}\\
      1-\sum_{p=1}^n\bar{q}_{ip}[k|\tau] & \text{if $i=j$}\\
      0 & \text{otherwise,}
    \end{cases}       
\end{equation}
where $\bar{\E}_{i,{k|\tau}}$ denotes the (pruned) edge set at node $i \in [n]$.

\paragraph{Pruned Graph based Averaging} 
For all iterations $t \in [k,k+\tau)$ with $k \in \mathcal{I}$, the algorithm performs decentralized averaging using the pruned weights, $\bar{q}_{ij}[k|\tau]$. Subsequent to this, the pruning step (Line 4, Algorithm~\ref{alg.ada_cons}) is performed again with the updated node estimates.

\begin{remark} We make the following remarks about Algorithm~\ref{alg.ada_cons}.
\begin{itemize}[leftmargin=0.75cm]

\item It is worth noting that the ideal choice of values for $\bar{\kappa}_i$ and $\underline{\kappa}_i$ can be problem-specific and depends on the network structure. 
For instance, preserving connectivity might be crucial in some cases, while in others, optimizing for low communication overhead may take precedence. Broadly speaking, a higher value of $\bar{\kappa}_i$ results in  aggressive pruning more suited to graphs with high edge density. Conversely, $\underline{\kappa}_i$ acts as a lower bound on the edges to be retained post pruning, and a higher value of $\underline{\kappa}_i$ corresponds to a more conservative pruning approach, which is beneficial if maintaining connectivity is important. For $\beta$, lower values lead to increased randomness in edge selection, resembling approaches such as the gossip protocol \cite{randomgossip}, while higher values promote a more deterministic and greedy approach to edge selection. 

\item If directed edges are permitted in the output of the pruning protocol, the application of the push-sum protocol \cite{pushsum} offers an alternative to simple distributed averaging that 
alleviates the requirement for doubly stochastic mixing matrices.
\end{itemize}
\end{remark}

\subsection{Convergence Analysis}

To provide convergence guarantees, we begin by writing the key step of \texttt{AC} (Line 7, Algorithm~\ref{alg.ada_cons}) in matrix form by employing the stacked vector notation,  
\begin{equation}\label{ac-algo}
    \textbf{x}_{k+1} = \textbf{Q}_{k} \textbf{x}_{k}, 
\end{equation}
where $\textbf{Q}_{k} = Q_k\otimes I_d= Q_{k|\tau}\otimes I_d \in \mathbb{R}^{nd \times nd}$, where $Q_{k|\tau}  := [q_{ij}[k|\tau]]_{i\in [n],j\in [n]} \in \mathbb{R}^{n \times n}$ denotes the mixing matrix of the pruned graph $\G(\V,\bar{\EB}_{k|\tau})$ for the $k|\tau$ cycle. We use  $\Q[r:s] \in \mathbb{R}^{nd \times nd}$ to denote the product of $s-r$ consecutive matrices indexed by $\{\Q_k\}_{k=r}^{s-1}$, i.e., $\textbf{Q}[r:s] := \Q_{s-1}\times\cdots \times \Q_{r}$, 
with the convention that $Q[s:s]:=I_n\otimes I_d \in \mathbb{R}^{nd \times nd}$. Using the above notation, we can express $\x_{(k+1)\tau}$ for any $k\geq 0$ in terms of $\x_0$ as follows
\begin{equation}\label{ac-algo-1}
\begin{aligned}
\textbf{x}_{(k+1)\tau} =\textbf{Q}_{k|\tau}^{\tau}\textbf{x}_{k\tau}&=   \Q[k\tau : (k+1)\tau]\textbf{x}_{k\tau} \\&= \Q[k\tau : (k+1)\tau] \times \cdots \times \Q[0 : \tau]\textbf{x}_{0}.
\end{aligned}
\end{equation}
We establish convergence under the following assumption.
\begin{assumption}[$\ta$\textit{-Connectivity}]\label{asmp2}  There exists a constant $\ta \in \mathbb{N}$, such that for all $k \in \mathcal{I}_{\ta}:= \{\ta,\ta+\tau, \ta+2\tau\cdots\}\subset \mathcal{I}$, the graph $\G(\V,\bar{\EB}_{(k|\tau -\ta+1)}) \cup \cdots \cup \G(\V,\bar{\EB}_{k|\tau})$ is connected.
\end{assumption}
\begin{remark}
Assumption~\ref{asmp2} plays a key role in the analysis. 
In words, it implies the existence of a constant $\ta$, such that within $\ta$ pruning cycles, the union of the resulting undirected (directed) pruned graphs is connected (strongly connected). For the special case where the pruned graph is connected for all cycles, 
$\ta=1$. It is possible to guarantee this assumption by imposing a consensus iteration with the reference graph every $\ta$ iterations of the algorithm for some finite $\ta \in \mathbb{N}$. Additionally, it is worth noting that it suffices to assume this property only for indices $\mathcal{I}_{\ta}$ rather than for all $k \in \mathbb{N}$. Another important point to note is that the assumption can be replaced by a stochastic version which takes into account the utilization of softmax based sampling in the pruning protocol. Specifically, the assumption of connectedness can either be assumed to hold almost surely or replaced by an assumption that ensures 
a reduction in the consensus error in expectation (with respect to $\Q_k$) over a period of $\ta$ iterations.
\end{remark}

To prove convergence of the algorithm, we need to establish convergence of the following product sequence to the $\frac{1}{n}\textbf{1}_{n}\textbf{1}_{n}^T$ rank-one matrix, i.e., 
\begin{align*}
    \prod_{j=0}^k \Q[j\tau : (j+1)\tau] \to \tfrac{1}{n}\textbf{1}_{n}\textbf{1}_{n}^T, \qquad\text{ as } k\to \infty.  
\end{align*}
To show this, we use the notion of coefficient of ergodicity \cite{sen}, denoted by $\rho(Q)$ for any row-stochastic matrix $Q$, defined as,
\begin{equation}\label{rho-eq}
\rho(Q) := 1 - \min_{i_1,i_2}  \sum_{j=1}^n \min \left( q_{i_1 j},q_{i_2j}  \right).
\end{equation}
Using the coefficient of ergodicity instead of directly bounding the spectral gap offers several advantages, particularly in scenarios involving time-varying topologies. First, it allows us to clearly 
characterize 
the influence of 
different graph parameters, such as maximum node degree and diameter, on convergence. This characterization helps us establish an explicit relationship between pruning and convergence. 
Second, it allows for extensions to 
directed graphs (with push-sum protocols) where the condition of double stochasticity may not be satisfied.

There are two key properties of~\eqref{rho-eq} that will be useful in establishing convergence. The first property is that $\rho(\cdot)$  is sub-multiplicative, i.e., for any two matrices $Q_1,\,Q_2$, 
\begin{equation}\label{cross}
\rho(Q_1Q_2) \leq \rho(Q_1) \rho(Q_2).
\end{equation}
The second property is that it can serve as an upper bound on the dissimilarity between the rows of matrix $Q$. More formally, we have (cf. \cite[Lemma 2]{ergo1}, \cite[Lemma 4]{ergo2})
\begin{equation}\label{upbound}
 \delta(Q) :=  \max_{j} \max_{i_1,i_2} |q_{i_1j} - q_{i_2j} | \leq \rho(Q),
\end{equation}
for any matrix $Q$ which is ergodic, i.e., it is row stochastic, aperiodic and irreducible (cf. \cite{ergo1} or, \cite[Chapter 8]{matbook}).  

Next, we state and prove the main theoretical result of this section.

\begin{thm}\label{thm0}
Suppose that: $(i)$ Assumptions~\ref{asmp1} and \ref{asmp2} hold, $(ii)$ the matrices $Q_k:= [q_{ij}[k]]_{i\in[n],j\in[n]}$ are doubly stochastic for all $k\geq 0$, $(iii)$ $q_{ii}[k]>0$ for all $k\geq 0$ for at least one $i \in [n]$, and, $(iv)$ if $q_{ij}[k]>0$ for any $(i,j) \in \E$ and $k\geq 0$, then $q_{ij}[k]  >q $ for some strictly positive constant $q > 0$ independent of $k$ and $(i,j)$. Then, for any $k \geq 0 $,
\begin{equation}\label{mainres}
    \|\x_k - \bar{\x}_k \| \leq n^{\frac{3}{2}} \gamma^{ \left\lfloor \frac{k}{\ta d_\G}\right\rfloor } \|\x_0-\bar{\x}_0\|,
\end{equation}
where $\gamma := \left( 1- q ^{\ta d_\G}\right) <1$ with $q<1$ and $d_\G$ is the diameter of a graph $\G(\V,\E)$ 
defined as $d_\G := \max_{u,v \in \V}\{\text{dist}(u,v)\}$, where $\text{dist}(u,v)$ denotes the shortest path distance between any two vertices $u,\,v\in \V$.
\end{thm}
\begin{proof}
We first establish the ergodicity of the product sequence $\mathbf{Q}[m\ta:(m+1)\ta] $ for any $m\geq 0$ with $\ta \in \mathbb{N}$ as in Assumption \ref{asmp2}. The stochasticity of $\mathbf{Q}[m\ta:(m+1)\ta] $ follows from that the fact that the product of  stochastic matrices is also stochastic. Furthermore, a matrix is considered irreducible if its zero/non-zero structure corresponds to a connected graph. By Assumption \ref{asmp2}, the structure of $\mathbf{Q}[m\ta:(m+1)\ta]$ also exhibits this property \cite[Section 1-C]{hadj}. Finally, an irreducible matrix is aperiodic if it has at least one self-loop which is satisfied by $\mathbf{Q}[m\ta:(m+1)\ta]$ by condition $(ii)$ in the theorem statement \cite[Section 1-C]{hadj}.

Next, we establish a useful upper bound on $\delta\left(\mathbf{Q}[0:k+1]\right)$. To do this, we consider the following decomposition of $\Q[0:k+1]$ 
\begin{align*}
&\Q[0:k+1] \\
= &\underbrace{\Q[0: \bar{k}\ta]\times \cdots\times \Q[m\bar{k}\ta: (m+1)\bar{k}\ta] \times\cdots \times\Q[ (K-1) \bar{k}\ta : K \bar{k}\ta]}_{\mathbf{Q}[0:K\bar{k}\ta]}\times 
 \Q[K\bar{k}\ta :k+1]
\end{align*}
where $K := \lfloor k/ \bar{k} \ta \rfloor$, $\bar{k}\geq 1$ is a constant to be specified later. Let $\tau' := \bar{k}\ta$. We bound $\delta\left(\mathbf{Q}[0:K\tau']\right)$ by individually bounding $\rho\left( \mathbf{Q}[m\tau':(m+1)\tau']\right)$ in the above product. 
By \eqref{rho-eq}, it follows that, 
\begin{equation}\label{rhodef}
    \rho\left( \mathbf{Q}[m\tau':(m+1)\tau']\right) =  1 - \min_{i_1,i_2}  \sum_j \min \left( q_{i_1 j}[m\tau':(m+1)\tau'],q_{i_2j}[m\tau':(m+1)\tau'] \right),
\end{equation}
where $\Q [m\tau':(m+1)\tau'] := [q_{ij}[m\tau':(m+1)\tau']]_{i,j\in [n]}$. By \eqref{rhodef}, we note that $\rho\left( \mathbf{Q}[m\tau':(m+1)\tau']\right)$ is guaranteed to satisfy $\rho\left( \mathbf{Q}[m\tau':(m+1)\tau']\right)<1$, if for every pair of rows $i_i$ and $i_2$, there exists some $j^*$ such that $q_{i_1 j^*}[m\tau':(m+1)\tau'],\, q_{i_2j^*}[m\tau':(m+1)\tau']>0$, i.e., if there is a path from some $j^*$ to both $i_1$ and $i_2$. This, in turn, is always satisfied if for some $\bar{k}>0$, $ q_{i j}[m\tau':(m+1)\tau']>0$ for every  $i,\,j \in [n]$,  i.e., all the entries are strictly positive. 

To find such a candidate $\bar{k}$, we make the following observation: $\Q[m\ta:(m+1)\ta]$ is ergodic, so there exists a path from $i$ to $j$ for every $i,\,j \in [n]$. Setting $\bar{k}=d_\G$ in the definition of $\tau'$, we have $\tau'=\bar{k}\ta = d_\G \ta$. It follows that for the matrix $\Q[m\tau':(m+1)\tau']$, $ q_{i j}[m\tau':(m+1)\tau']>0$ for all $i,\,j \in [n]$ since we can reach any node $i$ from any other node $j$ in at most $\tau' = d_\G \ta$ steps. 

For the remainder of the proof, let $\tau'=d_\G\ta$.
To lower bound $ q_{i j}[m\tau':(m+1)\tau']>0,\, m\geq0$, we note that by the definition of $q$ and Assumption \ref{asmp2}, it follows that $q^{}_{ij}[p\ta:(p+1)\ta] \geq q^{\ta}
$ for any $p\geq 0$ and any $(i,j) \in \E_{ p\ta} \cup \cdots \cup\E_{(p+1)\ta-1} $. Since $\Q[m\tau':(m+1)\tau'] = \Q[m\tau':m\tau'+\ta] \cdots \Q[m\tau'+(d_\G-1)\ta:m\tau' + d_\G\ta]$, for any $i',j'\in [n]$,
\begin{align}\label{upbound02}
    q_{i'j'} [m\tau':(m+1)\tau'] \geq q^{ \ta d_\G} .
\end{align}
By \eqref{rhodef} and \eqref{upbound02}, 
\begin{align}\label{upbound0}
    \rho\left( \mathbf{Q}[m\tau':(m+1)\tau']\right) &\leq  1- q^{ \ta d_\G} .
\end{align}
Thus, it follows that, 
\begin{align}\label{finalbd}
\delta(\Q[(0:K\ta]) &\leq  \rho\left( \mathbf{Q}[0:K\ta]\right) \nonumber\\
&\leq \rho( \mathbf{Q}[0: \ta]\cdots \mathbf{Q}[(K-1)\ta:K\ta])\nonumber\\
&\leq \rho( \mathbf{Q}[0: \ta])\times\cdots\times \rho(\mathbf{Q}[(K-1)\ta:K\ta])\nonumber\\
&\leq  \left( 1- q^{  \ta d_\G }\right)^{K},
\end{align}
where the first inequality follows by \eqref{upbound}, the second inequality by the the sub-multiplicative property of $\rho(\cdot)$ \eqref{cross}, and the final inequality follows by \eqref{upbound0}.
By \eqref{ac-algo} and \eqref{ac-algo-1}, it follows that,
\begin{align}\label{a1}
\textbf{x}_{k}  &= \Q_{k-1} \mathbf{x}_{k-1} \nonumber\\
&= \Q[K\ta+1:k]\Q[(K-1)\ta:K\ta] \times \cdots \times\Q[0:\ta] \textbf{x}_{0} \nonumber \\
&= \Q[K\ta+1:k]\Q[0:K\ta] \textbf{x}_0.
\end{align}
Multiplying both sides of \eqref{a1} by $\textbf{1}_{n}$, by the the double stochasticity of $\Q[K\ta+1:k]$ and $\Q[0:K\ta]$, it follows that,
\begin{align}\label{a2}
    \bar{\x}_{k}  =\bar{ \x}_0 &= \tfrac{1}{n} \textbf{1}_n\textbf{1}_n^k \x_0
= \tfrac{1}{n} \textbf{1}_n\textbf{1}_n^k   \Q[0:K\ta]  \textbf{x}_0. 
\end{align}
Subtracting \eqref{a2} from \eqref{a1}, 
\begin{align*}
\textbf{x}_{k}  -\bar{\x}_{k}  &=   \Q[K\ta+1:k]\Q[0:K\ta] \textbf{x}_0 -\tfrac{1}{n}  \textbf{1}_n\textbf{1}_n^k  \Q[0:K\ta]  \textbf{x}_0 \\
 &=  \Q[K\ta+1:k] 
 \left(\Q[0:K\ta] - \tfrac{1}{n}  \textbf{1}_n\textbf{1}_n^k \Q[0:K\ta]\right)  (\textbf{x}_0  -\bar{\x}_{0}),
\end{align*}
where the second equality holds due to $\Q[K\ta+1:k] \textbf{1}=\textbf{1}$  and the fact that $\textbf{A}\bar{\x}_0 = \bar{\x}_0$ for any doubly stochastic matrix $\textbf{A}$.
Taking norms of the above, it follows that,
\begin{align}\label{alteq}
\| \textbf{x}_{k}  -\bar{\x}_{k} \|  &\leq  \left\| \Q[K\ta+1:k] \right\| \left\|
 \Q[0:K\ta] - \tfrac{1}{n} \textbf{1}_n\textbf{1}_n^k \Q[0:K\ta] \right\| \|\textbf{x}_{0}-\bar{\x}_{0}\|\nonumber\\
 &\leq \sqrt{n}\left\|
 \Q[0:K\ta] - \tfrac{1}{n} \textbf{1}_n\textbf{1}_n^k \Q[0:K\ta] \right\|_1 \|\textbf{x}_{0}-\bar{\x}_{0}\|_1
\end{align}
where the first inequality is due to the Cauchy–Schwarz inequality and the second inequality follows due to the facts that $ \|A\| \leq \sqrt{n} \|A\|_1$ for any $A\in \mathbb{R}^{n\times n}$ and $ \left\| \Q[K\ta+1:k] \right\| \leq 1$. 
We have by definition of the $\ell_1$-norm for matrices,
\begin{align}
    \left\| \mathbf{Q}[0:K\ta]  - \tfrac{1}{n}\textbf{1}_n\textbf{1}^T_n \mathbf{Q}[0:K\ta] \right\|_1  &= \max_{1\leq j \leq n} \sum_{i=1}^n \left|q_{ij}[0:K\ta]  - \tfrac{1}{n}\sum_{i'=1 }^nq_{i'j}[0:K\ta]  \right|\nonumber\\
    &\leq \max_{1\leq j \leq n}  \sum_{i=1}^n  \tfrac{1}{n}\sum_{i'=1}^n \underbrace{\Big| q_{ij}[0:K\ta] -  q_{i'j}[0:K\ta]  \Big|}_{\leq \delta(\mathbf{\bar{Q}}[0:K\ta] )}\nonumber\\
    &\leq  n \delta(\mathbf{\bar{Q}}[0:K\ta] )\label{OI}\\
    &\leq  n\left( 1- q^{  d_\G \ta}\right)^{K}\label{I},
\end{align}
where the last inequality follows by \eqref{finalbd}. Combining \eqref{I} and \eqref{alteq} with $K = \lfloor k/ \bar{k} \ta \rfloor$ completes the proof.
\end{proof}

We note that the convergence rate in Theorem \ref{thm0} is primarily dependent of the diameter of the graph, $d_\G$, and the lower bound on the nonzero entries of the mixing matrix, $q$. The form of the convergence rate factor $\gamma$ confirms the empirical observation that compact graphs with shorter diameters generally fare better with pruning since multiple information pathways can potentially exist between two nodes. The dependence on $q$ can be illustrated by considering the Metropolis-Hastings scheme as described in \eqref{MH}. Let $n_{\G_{k|\tau}}$ denote the maximum node degree of graph $\G(\V,\EB_{k|\tau})$. If
 $n_{\text{max}} := \max_{k\in \mathcal{I}} n_{\G_{k|\tau}} $, denotes the maximum node degree amongst all the pruned graphs (assumed to be connected) obtained during the algorithm, then $q = \frac{1}{1+n_\text{max}}$. Since $n_{\text{max}}$ can be smaller than the maximum node degree of the underlying reference graph, $q$ can potentially be larger for \texttt{AC}. 

\begin{remark} \label{remnew}
We make the following additional remarks about Theorem \ref{thm0}.
\begin{itemize}[leftmargin=0.75cm]
    \item It should be noted that the convergence factor $\gamma$ in (\ref{mainres}) may be 
a conservative estimate in general. Nevertheless, the analysis provided here remains applicable in a broad range of scenarios, even when tighter estimates for specific cases may not hold. In particular, the extension of Theorem \ref{thm0} to a directed graph setting, where only column stochasticity is satisfied (as in the push-sum protocol), can be derived 
relatively easily. This is due to the fact that the definition of the coefficient of ergodicity and the associated bounds, e.g., \eqref{upbound}, do not necessitate a double-stochasticity assumption on the matrix $Q_k$.
    \item The assumptions on the matrix entries of $Q_k$ in Theorem \ref{thm0} are typical in ergodic matrix literature \cite{matbook} and multi-agent coordination and optimization problems \cite{survey}. For undirected graphs, the assumptions are satisfied if the weights are generated according to 
    \eqref{MH}.
    \item To understand (and quantify) 
    the impact of pruning on distributed averaging within a simplified context, let us consider a scenario where there is a total communication budget of $B$ bits, and each node utilizes $D$ bits to transmit the quantized objective variable to its neighboring nodes. The maximum number of iterations that can be executed under these settings is given by $T = \frac{B}{2D|\E|}$. Let $\sigma(Q)$ denote the spectral gap of the mixing matrix $Q$, assumed to be generated in accordance to \eqref{MH}. Under Assumption \ref{asmp1}, for $x_k$ generated via \eqref{ac-algo} with $\Q_k=Q\otimes I_d,\,\forall \,k$, 
\begin{equation}\label{pedagog1}
\|\x_{T^{}} - \bar{\x}_{T^{}} \| \leq (1- \sigma(Q))^T\|\x_0 - \bar{\x}_0\|.
\end{equation}
If we consider the same scenario with a fraction $\kappa<1$ of the edges pruned (where the pruned mixing matrix is denoted by $Q^{\text{prune}}$) and assume the pruned graph satisfies Assumption \ref{asmp1}, we have\footnote{To keep the presentation clear, we assume $T,T^{\text{prune}} \in \mathbb{N}$.},
\begin{equation}\label{pedagog2}
\|\x_{T^{\text{prune}}} - \bar{\x}_{T^{\text{prune}}} \| \leq \left(1- \sigma(Q^{\text{prune}})\right)^{T^{\text{prune}}}\|\x_0 - \bar{\x}_0\|.
\end{equation}
Since $T^{\text{prune}} = \frac{B}{2(1-\kappa)D|\E|}=\frac{T}{1-\kappa}> T$, the upper bound for the consensus error with the pruned network, where $\sigma(Q^{\text{prune}})\approx \sigma(Q^{})$, is potentially tighter since $\left(1- \sigma(Q^{\text{prune}})\right)^{T^{\text{prune}}} \lessapprox\left(1- \sigma(Q^{})\right)^T$.
In Section \ref{sec4.1} (Figure \ref{fg1}(c)), we empirically observe that $ \sigma(Q^{\text{prune}})$ for small to medium values of $\kappa$ does not significantly deviate from $\sigma(Q)$, 
suggesting that there are instances for which the inequality is likely to hold.
\end{itemize}
\end{remark}

\section{Adaptive Consensus based Decentralized Optimization}\label{sec.acgt}
In this section, we describe the proposed Adaptive Consensus based Gradient Tracking algorithm (Algorithm~\ref{alg.acgt}, \texttt{AC-GT}) 
for decentralized optimization. The problem under consideration can be expressed as,
\begin{equation}\label{mainprob0}
\begin{aligned}
\min_{\x \in \mathbb{R}^{nd}} \quad&f(\x) :=\frac{1}{n} \sum_{i=1}^n f_i(x_i)\\
\text{s.t.} \quad & \Q\x = \x,
\end{aligned}
\end{equation}
where $f:\mathbb{R}^{nd}\to \mathbb{R}$ and $\Q:=Q\otimes I_d\in \mathbb{R}^{nd\times nd}$. Under Assumption \ref{asmp1}, the constraint is equivalent to the condition that $x_i=x_j ,\,\text{for all } \,i,\,j \in [n]$, and thus problems \eqref{mainprob0} and 
\eqref{mainprob} are equivalent. We make the following assumption with 
regards to the component functions ($f_i$).
\begin{assumption}[Regularity and convexity of $f_i$]\label{asmp3}  Each $f_i$ is $L$-smooth and $\mu$-strongly convex.\end{assumption}

The general idea of \texttt{AC-GT} is to leverage the adaptive consensus protocol of the previous section and combine it with a gradient tracking algorithm \cite{DIG} in a manner that preserves the strong convergence guarantees of the latter while harnessing the communication savings of the former. The pseudo-code for the algorithm is provided in Algorithm~\ref{alg.acgt}.

\begin{algorithm}[]
\footnotesize
   \caption{\small \texttt{ADAPTIVE CONSENSUS BASED GRADIENT TRACKING (AC-GT)}\label{alg.acgt}}
    \footnotesize
    \textbf{Inputs:} Graph $\mathcal{G}(\mathcal{V},\mathcal{E})$; Cycle Length $\tau \in \mathbb{N}$; Softmax parameter $\beta \in [0,\infty]$; Thresholding factors $(\bar{\kappa}_i,\ubar{\kappa}_i) \in [0,1]^2$ for all $i\in [n]$; Step size $\alpha>0$; Initial iterates $x_{i,0} \in \mathbb{R}^d$, $y_{i,0} = \nabla f_i(x_{i,0})$ for all $i\in [n]$; Total number of iterations $T\in \mathbb{N}$. 
  \begin{algorithmic}[1]
   \footnotesize
   \FOR{$k=0,\dots,T$}
     \FOR{ all $i\in [n]\,$ in parallel} 
          \IF{$k \in \mathcal{I},$}
           \STATE Generate $\mathcal{G}(\mathcal{V},\bar{\EB}_{k|\tau}) \sim$ \texttt{PRUNING PROTOCOL}($\mathcal{G}(\mathcal{V},\mathcal{E}),x_{i,k},(\bar{\kappa}_i,\ubar{\kappa}_i),\beta)$.
           \STATE Get new weights $\bar{q}_{ij}[k|\tau] \sim $ \texttt{GENERATE WEIGHTS} ($\mathcal{G}(\mathcal{V},\bar{\EB}_{k|\tau})$).
            \STATE Generate $\mathcal{G}(\mathcal{V},\hat{\EB}_{k|\tau}) \sim$ \texttt{PRUNING PROTOCOL}($\mathcal{G}(\mathcal{V},\mathcal{E}),y_{i,k},(\bar{\kappa}_i,\ubar{\kappa}_i),\beta)$.
           \STATE Get new weights $\hat{q}_{ij}[k|\tau] \sim $ \texttt{GENERATE WEIGHTS} ($\mathcal{G}(\mathcal{V},\hat{\EB}_{k|\tau})$).
           \ENDIF
         \STATE Update estimate at node $i$ according to:
          $
          x_{i,k+1} = \sum_{j=1}^n \bar{q}_{ij}[k|\tau]\left(x_{j,k} -\alpha y_{j,k}\right)
          $.
          \STATE Update gradient estimate at node $i$ according to:
          $
          y_{i,k+1} = \sum_{j=1}^n\hat{q}_{ij}[k|\tau] y_{j,k} + \nabla  f_i(x_{i,k+1}) - \nabla f_i(x_{i,k}) 
          $.
         
       \ENDFOR     
    \ENDFOR
\end{algorithmic}
 \textbf{Output:} $x_{i,T}$ for all $i \in [n]$.
\end{algorithm}
To provide intuition for the algorithm, we review the main steps of the gradient tracking algorithm (\texttt{GTA}), as it serves as a foundational component of \texttt{AC-GT}. The main iterations of the gradient tracking algorithm can be expressed as,
\begin{align*}
    x_{i,k+1} &=  \sum_{j=1}^n q_{ij}\left( x_{j,k} -\alpha y_{j,k}\right),
\,\qquad\,   y_{i,k+1} =  \sum_{j=1}^n q_{ij}y_{j,k} +\nabla f_i(x_{i,k+1}) - \nabla f_i(x_{i,k}),
\end{align*}
where $\alpha>0$ is a constant referred to as the step size.

The underlying computational principles of \texttt{AC-GT} are similar to those of \texttt{GTA}. However, the communication structure of \texttt{AC-GT} is based on \texttt{AC}. Similar to \texttt{AC}, \texttt{AC-GT} operates in a cyclical manner. In the $k|\tau$ cycle, if $k$ belongs to the set $\mathcal{I}$, the pruning protocol is executed twice. The first instance employs the $\x$ estimates to get the pruned graph ($Q_k$) and the associated mixing matrix, which are subsequently utilized to update the $\x$ estimate,
\begin{align}
    \textbf{x}_{k+1} &= \textbf{Q}_{k} \left(  \textbf{x}_{k} -\alpha \textbf{y}_k\right),  \text{ where }\; \textbf{Q}_{k} =   \Q_{k|\tau} ,\,\,\forall k\in \left[(k|\tau)\tau,\, (k|\tau+1)\tau \right). \label{x-iter}
\end{align}
The second instance of the protocol obtains a different pruned graph ($\hat{Q}_k$) using the $\y$ estimates. The mixing matrix corresponding to this graph is then used to update the $\y$ estimate as follows,
\begin{align}
    \textbf{y}_{k+1} &= \hat{\textbf{Q}}_{k} \textbf{y}_{k} + \nabla \textbf{f}(\textbf{x}_{k+1}) - \nabla \textbf{f} (\textbf{x}_{k}), \text{ where } \; \hat{\Q}_{k}\ =    \hat{\Q}_{k|\tau},\ \,\,\forall k\in \left[(k|\tau)\tau,\, (k|\tau+1)\tau \right) . 
\label{y-iter}
\end{align}
The pruning protocol is executed twice because the dissimilarity between the $\y$ estimates is expected to be different from the dissimilarity between the $\x$ estimates. \texttt{AC-GT} employs a constant step size $\alpha>0$ which depends on both the properties of the function and the structure of the pruned network as shown in the next subsection.

\begin{remark} We make a couple of remarks about \texttt{AC-GT} (Algorithm~\ref{alg.acgt}).
\begin{itemize}[leftmargin=0.75cm]
    \item For $\tau=1$ and $\Q_k=Q\otimes I_d$ for all $k \in \mathbb{N}$, where $Q$ is the mixing matrix corresponding to the reference 
    graph, \texttt{AC-GT} reduces to a standard gradient tracking algorithm (\texttt{GTA}) \cite{DIG}.
    \item The extension of \texttt{AC-GT} to a directed graph setting is feasible 
    by leveraging the push-pull gradient algorithm \cite{pupush}. Similar to \texttt{AC}, the principles and theory of \texttt{AC-GT} for the directed graph setting can be derived from the current framework, with appropriate adjustments.
\end{itemize}
\end{remark}

\subsection{Convergence Analysis}
We provide theoretical convergence
guarantees for \texttt{AC-GT}. For simplicity, we assume that $\Q_k=\hat{\Q}_k$ for all $k \geq 0$ in \eqref{x-iter} and \eqref{y-iter} 
and note that one can derive the same results verbatim for the case where $\Q_k\neq\hat{\Q}_k$, with additional notation required. We build up to our main result through a series of technical lemmas which we state next. We begin by proving a descent relation for the consensus error $\Psi_k$, defined as,
\begin{equation}\label{eq.psi_descent}
    \Psi_k:= \begin{bmatrix}
 \x_{k} -\bar{\textbf{x}}_{k} \\
\alpha  (\y_{k} -\bar{\textbf{y}}_{k} )
\end{bmatrix}\in \mathbb{R}^{2nd}.
\end{equation}
\begin{lem}\label{lem1}
Suppose that the matrices $\Q_k$, for all $k$, are 
doubly stochastic and $\hat{\Q}_k=\Q_k$. 
For $\Psi_k$ given in \eqref{eq.psi_descent} and $\tap \in \mathbb{N}$,   
    \begin{align}
  \|\Psi_k \|^2 &\leq  \rho' \|\Psi_{k-\tap}\|^2 +b \sum_{j=k-\tap}^{k-1}\|\Psi_{j}\|^2 +  c\sum_{j=k-\tap}^{k-1} \left(f(\bar{x}_{j}) -f(x^*)\right),\quad \mbox{if } \;\; k\geq\tap, \label{rbound} \\
    \|\Psi_k \|^2 &\leq 5(1+\tap^2) \|\Psi_{0}\|^2 +b \sum_{j=0}^{k-1}\|\Psi_{j}\|^2 +  c\sum_{j=0}^{k-1} \left(f(\bar{x}_{j}) -f(x^*)\right), \quad \mbox{if }\;\;  0<k<\tap,  \label{rbound1}
    \end{align}
    where $\rho' := 2(1+\tap^2)\max_{\tap\leq j\leq t}  \left\| \Q[j-\tap :j] -  \frac{1}{n} \textbf{1}_n\textbf{1}_n^T \right\|^2$, $b:= 180\alpha^2 L^2 (1+\tap^2)\tap$, and $c:= 320n\alpha^4 L^3(1+\tap^2)\tap$.
\end{lem}
\begin{proof}
We start by considering the expression $\x_{k} -\bar{\textbf{x}}_{k}$. By \eqref{x-iter} and the double stochasticity of $\Q_k$, 
\begin{align}\label{x-con}
    \x_{k} -\bar{\textbf{x}}_{k} &= \left(\Q_{k-1} - \tfrac{\textbf{1}_n\textbf{1}_n^T}{n}\right)\left(\x_{k-1} - \bar{\x}_{k-1}- \alpha (\textbf{y}_{k-1}- \bar{\textbf{y}}_{k-1}) \right). 
\end{align}
Using \eqref{y-iter}, a similar expression for $\mathbf{y}_{k} -\bar{\textbf{y}}_{k}$ is given as,
\begin{equation}\label{y-con}
\begin{aligned}
    \mathbf{y}_{k} -\bar{\textbf{y}}_{k} &=  \left(\Q_{k-1} - \tfrac{\textbf{1}_n\textbf{1}_n^T}{n}\right)\left(\textbf{y}_{k-1} - \bar{\y}_{k-1} \right)  -\left(\textbf{I}_{n}- \tfrac{\textbf{1}_n\textbf{1}_n^T}{n}\right) (\nabla \textbf{f}(\x_{k})-\nabla \textbf{f}(\x_{k-1}) ) ,
\end{aligned}
\end{equation}
where $\textbf{I}_{n} := I_n \otimes I_{d} \in \mathbb{R}^{nd\times nd }$. 
The expressions in \eqref{x-con} and \eqref{y-con} can be compactly represented in matrix form as follows,
\begin{align}\label{conerr}
    \Psi_k  &= \J_{k-1} \Psi_{k-1} + \alpha \EE_{k-1} \nonumber\\
    &= \J[k-\tap :k] \Psi_{k-\tap} + \alpha \sum_{j=1}^{\tap} \J[k-j+1:k] \EE_{k-j},
\end{align}
where
\begin{equation}
\begin{aligned}\label{not1}
\J_k:=
\begin{bmatrix}
 \Q_{k}-  \tfrac{\textbf{1}_n\textbf{1}_n^T}{n}  & -\left(\Q_{k}-  \tfrac{\textbf{1}_n\textbf{1}_n^T}{n}\right)  \\
 0 &  \Q_{k}-  \tfrac{\textbf{1}_n\textbf{1}_n^T}{n}
\end{bmatrix}, 
\EE_{k-1} := 
\begin{bmatrix}
0\\
 \Big(\textbf{I}_{n} - \tfrac{\textbf{1}_n\textbf{1}_n^T}{n} \Big) (\nabla \textbf{f}(x_{k-1}) - \nabla \textbf{f}(x_{k} ))
\end{bmatrix}
\end{aligned}
\end{equation}
and $\J[k-j :k] := \J_{k-1} \cdots \J_{k-j}$, for any $j\leq \tap\leq k$. The matrix $\J[k-j :k] $ can be expressed as,
\begin{align}\label{jdef}
    \J[k-j :k] = \begin{bmatrix}
\Q[k-j :k]- \tfrac{\textbf{1}_n\textbf{1}_n^T}{n} & - j \left(\Q[k-j :k]- \tfrac{\textbf{1}_n\textbf{1}_n^T}{n}\right)  \\
 0 &  \Q[k-j:k] - \tfrac{\textbf{1}_n\textbf{1}_n^T}{n}
\end{bmatrix}. 
\end{align}
The above equation can be derived by a straightforward induction argument using the facts that
\begin{equation*}
\begin{bmatrix}
 A_{1} & -A_{1}  \\
 0 &  A_{1}
\end{bmatrix} \times\begin{bmatrix}
 A_{2} & -A_{2}  \\
 0 &  A_{2}
\end{bmatrix} = \begin{bmatrix}
 A_1A_2 & -2A_1A_2 \\
 0 &  A_1A_2
\end{bmatrix},
\end{equation*}
and, for any two doubly stochastic matrices $\Q$ and $\Q'$, 
\begin{equation*}
 \left(\Q - \tfrac{\textbf{1}_n\textbf{1}_n^T}{n}\right) \left(\Q' - \tfrac{\textbf{1}_n\textbf{1}_n^T}{n}\right) =  \left(\Q\Q' - \tfrac{\textbf{1}_n\textbf{1}_n^T}{n}\right).
\end{equation*}
By \eqref{jdef}, it follows that 
\begin{align}\label{c-}
\| \J[k-j :k]  \|^2 &\leq (1+ j^2)\left\| \Q[k-j :k] - \tfrac{1}{n} \textbf{1}_n\textbf{1}_n^T \right\|^2,
\end{align}
and, since $\left\| \Q[k-j :k-1] - n^{-1} \textbf{1}_n\textbf{1}_n^T \right\|^2\leq 4$, 
\begin{equation}\label{c2}
    \| \J[k-j :k]\|^2  \leq 4(1+j^2)\leq 4(1+\tap^2), \quad \forall \,j<\tap.
\end{equation}
Taking the norm square of \eqref{conerr},
\begin{align}\label{conerr1}
\|\Psi_k \|^2 &= \left\|\J[k-\tap :k]\Psi_{k-\tap}+ \alpha \sum_{j=1}^{\tap} |\J[k-j+1:k] \EE_{k-j}\right\|^2 \nonumber\\
&
\leq \left(1+\tfrac{1}{4}\right)\|\J[k-\tap :k]\Psi_{k-\tap}\|^2 +5 \alpha^2\left\| \sum_{j=1}^{\tap}\J[k-j+1:k] \EE_{k-j}\right\|^2 \nonumber\\
&\leq \tfrac{5}{4}(1+\tap^2) \left\| \Q[k-\tap :k] - \tfrac{1}{n} \textbf{1}_n\textbf{1}_n^T \right\|^2 \|\Psi_{k-\tap}\|^2 + 20\alpha^2 (1+\tap^2) \tap  \sum_{j=1}^{\tap} \|\EE_{k-j}\|^2, 
\end{align}
where the first inequality is due to the fact that $\| a+b\|^2 \leq (1+\xi)\|a\|^2 + (1+\xi^{-1})\|b\|^2 $ for any constant $\xi>0$, and the second inequality follows by \eqref{c-} with $j=\tap$, \eqref{c2}, and the fact that $\left\|\sum_{j=1}^{\tap} a_j\right\|^2\leq \tap \sum_{j=1}^{\tap}\|a_j\|^2$
. 
We next bound $\|\EE_{p-1}\|$ for any $p\geq 1$. By the definition of $\EE_{k}$  \eqref{not1} with $k=p$,
\begin{equation}\label{ek}
   \|\EE_{p-1}\|^2 \leq
 \left\|\Big(\textbf{I}_n- \tfrac{\textbf{1}_n\textbf{1}_n^T}{n} \Big) (\nabla \textbf{f}(\x_{p}) - \nabla \textbf{f}(\x_{p-1} )) \right\|^2\leq \left\|\nabla \textbf{f}(\x_{p}) - \nabla \textbf{f}(\x_{p-1} )) \right\|^2.
\end{equation}
The term on the right-hand-side of \eqref{ek} can be bounded as follows
\begin{align}\label{l1}
&\|  \nabla \textbf{f}(\x_{p}) -  \nabla \textbf{f}(\x_{p-1} )\|^2 \nonumber\\
\leq&  L^2\|\x_{p} - \x_{p-1}\|^2  \nonumber\\
=&  L^2\| (\Q_{p-1} - \mathbf{I}_n ) (\x_{p-1}-\bar{\textbf{x}}_{p-1}) -\alpha \Q_{p-1}\y_{p-1})\|^2  \nonumber\\
\leq& 2L^2\|  (\Q_{p-1} - \mathbf{I}_n ) \left(\x_{p-1} - \bar{\textbf{x}}_{p-1}\right)\|^2 + 2\alpha^2 L^2\|\textbf{y}_{p-1} \|^2 \nonumber\\
\leq& 8L^2\|  \x_{p-1} - \bar{\mathbf{x}}_{p-1}\|^2 + 4\alpha^2 L^2\|\mathbf{y}_{p-1} -\bar{\mathbf{\y}}_{p-1}\|^2+4\alpha^2 L^2 \|\bar{\mathbf{\y}}_{p-1}\|^2,
\end{align}
where we have used Assumption \ref{asmp3} to get the first inequality, \eqref{x-iter} with $k=p-1$ to substitute for $\x_{p}$ and the fact that $(\Q_{p-1} - \mathbf{I}_n )\bar{\x}_{p-1} = 0$ to get the equality, and 
$\left\|  \Q_{p-1} -\mathbf{I}_n \right\| \leq 2$ to obtain the first term in the last inequality. By Assumption \ref{asmp3},
\begin{align}\label{l2}
    \|\bar{\textbf{y}}_{p-1} \|^2 =& n\|\bar{y}_{p-1}\|^2 \nonumber\\
    =&n \left\| \frac{1}{n} \sum_{i=1}^n \n f_i( x_{i,p-1}) \right\|^2 \nonumber\\
    \leq& 2n \left\| \frac{1}{n} \sum_{i=1}^n \n f_i( x_{i,p-1})  -\frac{1}{n} \sum_{i=1}^n  \n f_i (\bar{x}_{p-1}) \right\|^2 \nonumber\\
    &+ 2n \left\| \frac{1}{n} \sum_{i=1}^n \n f_i( \bar{x}_{p-1})  -\frac{1}{n} \sum_{i=1}^n  \n f_i (x^*) \right\|^2  \nonumber\\
    \leq& 2L^2\|\x_{p-1} - \bar{\textbf{x}}_{p-1}\|^2 + 4L\sum_{i=1}^n \left(f_i(\bar{x}_{p-1}) -f_i(x^*)\right).
\end{align}
Combining \eqref{ek}, \eqref{l1} and \eqref{l2}, and using the fact that 
$\alpha < 1/3L$, it follows that for any $p \geq 1$,
\begin{align}
    \| \EE_{p-1}\|^2 
    \leq& \| \nabla \textbf{f}(\x_{p}) - \nabla \textbf{f}(\x_{p-1} )\|^2 \nonumber\\
    \leq&    9L^2 \left( \|\x_{p-1} - \bar{\mathbf{x}}_{p-1}\|^2 + \alpha^2 \|\mathbf{y}_{p-1} -\bar{\mathbf{\y}}_{p-1}\|^2 \right) \label{ebound}\\
    & + 16\alpha^2L^3\sum_{i=1}^n \left(f_i(\bar{x}_{p-1}) -f_i(x^*)\right). \nonumber
\end{align}
Using \eqref{ebound} with $p=k-j+1$ to bound $\|E_{k-j}\|,\,1\leq j\leq \tap$ in \eqref{conerr1}, we get,
\begin{equation*}
\begin{aligned}
 \|\Psi_k \|^2 \leq&  2(1+\tap^2)\left\| \Q[k-\tap :k] - \tfrac{1}{n} \textbf{1}_n\textbf{1}_n^T \right\|^2 \|\Psi_{k-\tap}\|^2 \\
 & + 180\alpha^2 L^2 (1+\tap^2)\tap  \sum_{j=1}^{\tap} \|\Psi_{k-j}\|^2\\
 & 
 +  320n\alpha^4 L^3(1+\tap^2)\tap   \sum_{j=1}^{\tap} \left(f(\bar{x}_{k-j}) -f(x^*)\right), 
\end{aligned}
\end{equation*}
which proves \eqref{rbound}. To prove (\ref{rbound1}), we note that for $k<\tap$, we can write \eqref{conerr} as,
\begin{align}\label{conerr2}
    \Psi_k  &= \J_{k-1} \Psi_{k-1} + \alpha \EE_{k-1} 
    = \J[0 :k] \Psi_{0} + \alpha \sum_{j=0}^{k-1} \J[k-j:k] \EE_{j}.
\end{align}
Taking the norm square of \eqref{conerr2},
\begin{align}\label{conerr3}
    \|\Psi_k \|^2 &\leq \left(1+\tfrac{1}{4}\right)\|\J[0 :k]\Psi_{0}\|^2 +5 \alpha^2\left\| \sum_{j=0}^{k-1}\J[k-j:k] \EE_{j}\right\|^2 \nonumber\\
    &\leq 5(1+\tap^2) \|\Psi_{0}\|^2 + 20\alpha^2 (1+\tap^2) \tap  \sum_{j=0}^{k-1} \|\EE_{j}\|^2, 
\end{align}
where we have used 
$\| a+b\|^2 \leq (1+\xi)\|a\|^2 + (1+\xi^{-1})\|b\|^2 $ for any constant $\xi>0$ in the first inequality and \eqref{c2} to obtain the second inequality. The final result \eqref{rbound1} can be derived using \eqref{ebound} with $p=j+1$ for $1\leq j\leq k-1$ 
in \eqref{conerr3}.
\end{proof}

Next, we state an auxiliary lemma whose proof can be found in \cite[Lemma 4]{suh_rag}.
\begin{lem}\label{lem4}
Suppose the non-negative scalar sequences $\{a_t\}_{t\geq0}$ and $\{e_t\}_{t\geq0}$ 
satisfy the following recursive relation 
for a fixed $\tap\in \mathbb{N}$
\begin{align}\label{rel}
a_{t} &\leq \rho' a_{t-\tap} + \frac{b}{\tap} \sum_{i=t-\tap}^{t-1} a_{i}  + c  \sum_{i=t-\tap}^{t-1} e_{i} +r ,\qquad \mbox{if } \;\;t\geq \tap,\\
a_{t} &\leq \rho'' a_{0} + \frac{b}{\tap} \sum_{i=0}^{t -1} a_{i} + c  \sum_{i=0}^{t-1}  e_{i}  +r,\qquad \,\mbox{if } \;\; t<\tap ,\label{rel1}
\end{align}
where $b , \,c, \,r\,,\rho''$ are non-negative constants, 
$b \leq \rho'/4$ and $\rho'\in \left(0,1/4\right)$. Then, for any $t \in \mathbb{N}$,
\begin{align}\label{lem5_main}
a_t &\leq 20\rho''\Big(1- \frac{3\rho}{4\tap}\Big)^t a_0 + 60 c\sum_{i=0}^{t-1}   \Big(1- \frac{3\rho}{4\tap}\Big)^{t-i}e_i  +\frac{26r}{\rho}, 
\end{align}
where $\rho := 1-2\rho'$.
\end{lem}

We are ready to state and prove the main theorem.

\begin{thm}\label{thm2}
    Suppose that: $(i)$ Assumptions~\ref{asmp1} and \ref{asmp3} hold, and, $(ii)$ $\Q_k$ are doubly stochastic matrices and $\hat{\Q}_k=\Q_k$ for $k\geq 0$. Let $x_{i,k}$ denote the iterates generated via the recursions \eqref{x-iter}-\eqref{y-iter} and $\bar{x}_k:=n^{-1}\sum_{i=1}^nx_{i,k}$. Then, for all $k \geq 0$, 
    \begin{equation}
    \begin{aligned}\label{con-res}
         &\|\bar{x}_k - x^*\|^2 \\
         \leq& 
         \left(1-\tfrac{\alpha \mu}{4}\right)^k \left(\|\bar{x}_0-x^*\|^2+ \tfrac{1000L(1+\tap^2)}{\mu n\left(1-\frac{\alpha\mu}{4}\right)} \left(\|\x_0 -\bar{\x}_0\|^2+\alpha^2 \|\y_0 -\bar{\y}_0\|^2\right)  \right)
    \end{aligned}
    \end{equation}
where $\tap\in \mathbb{N}$ with  $\rho' := 2(1+\tap^2)\max_{\tap\leq t \leq k} \left\| \Q[t-\tap :t] - \frac{1}{n} \textbf{1}_n\textbf{1}_n^T \right\|^2<1/4$ and 
\begin{equation}\label{alpha_bound}
       \alpha<\min\left\{1,\tfrac{\sqrt{\rho'}}{58 L\tap^2}\right\}.
\end{equation} 
\end{thm}

\begin{proof}
By \eqref{x-iter}, the optimization error of the average iterates for any $t\in \mathbb{N}$ is
\begin{align}\label{ll2}
\|\bar{x}_{t+1} -x^*\|^2 &= \|\bar{x}_{t} -\alpha \bar{y}_t -x^* \|^2 \nonumber\\
&=  \left\| \bar{x}_{t} - \frac{\alpha}{n} \sum_{i=1}^n\n f_i (x_{i,t}) -x^* \right\|^2\nonumber\\
&=\| \bar{x}_t -x^* \|^2 -\frac{2\alpha}{n} \left\langle   \sum_{i=1}^n\n f_i (x_{i,t}), \bar{x}_{t}-x^* \right\rangle +\alpha^2 \left\|\frac{1}{n}  \sum_{i=1}^n\n f_i (x_{i,t}) \right\|^2,
\end{align}
where 
$\bar{y}_t =  n^{-1} \sum_{i=1}^n\n f_i (x_{i,t}) $. (This can be proven by an induction argument using (\ref{y-iter}).) 
The second term in (\ref{ll2}) can be bounded as, 
\begin{align}\label{bd2-}
     &\left\langle   \sum_{i=1}^n\n f_i (x_{i,t}), \bar{x}_{t}-x^* \right\rangle\nonumber\\
     =&  \left\langle   \sum_{i=1}^n\n f_i (x_{i,t}), \bar{x}_{t}-x_{i,t} \right\rangle +  \left\langle   \sum_{i=1}^n\n f_i (x_{i,t}), x_{i,t}-x^* \right\rangle \nonumber \\
     \geq&\sum_{i=1}^n \left[ f_i(\bar{x}_t) -f_i(x_{i,t})- \tfrac{L}{2}\|\bar{x}_t -x_{i,t}\|^2 + f_i(x_{i,t}) - f_i(x^*)+ \tfrac{\mu}{2}\|x_{i,t}-x^*\|^2\right] \nonumber\\
     \geq&\sum_{i=1}^n \Big[ f_i(\bar{x}_t) -f_i(x^*) -  \tfrac{L+\mu}{2} \|\bar{x}_t -x_{i,t}\|^2 +\tfrac{\mu }{4}\left\|\bar{x}_t -x^*\right\|^2\Big],
\end{align}
where Assumption \ref{asmp3} is used in the first inequality and the bound $ \|\bar{x}_t -x^*\|^2 \leq 2 \|\bar{x}_t-x_{i,t}\|^2 + 2\|x_{i,t} -x^*\|^2$ is used to derive the last inequality. The last term in \eqref{ll2} can be bounded as,
\begin{align}\label{bd1-}
     &\left\|\frac{1}{n}  \sum_{i=1}^n\n  f_i (x_{i,t}) \right\|^2 \nonumber\\
     =&  \left\|\frac{1}{n}  \sum_{i=1}^n\n f_i (x_{i,t})-\frac{1}{n}  \sum_{i=1}^n\n f_i (\bar{x}_{t})+\frac{1}{n}  \sum_{i=1}^n\n f_i (\bar{x}_{t})-\frac{1}{n}  \sum_{i=1}^n\n f_i (x^*) \right\|^2\nonumber\\
     \leq&  \frac{2L^2}{n} \sum_{i=1}^n \| x_{i,t}-\bar{x}_{t}\|^2+ \frac{4L}{n} \sum_{i=1}^n (f_i(\bar{x}_t) -f_i(x^*)),
\end{align}
where in the second summation we have used the fact that $\|\n f_i (\bar{x}_{t})  - \n f_i (x^*) \|^2 \leq 2L (f_i(\bar{x}_t) -f_i(x^*))$ by Assumption \ref{asmp3} \cite[Theorem 2.1.5]{yuri}. Using (\ref{bd2-}) and (\ref{bd1-}) in (\ref{ll2}) along with $\alpha<1/4L$, it follows that,
\begin{align}\label{ll01}
    \left\| \bar{x}_{t+1} -  x^*\right\|^2 \leq & \left(1-\tfrac{\alpha\mu}{2}\right) \| \bar{x}_t -x^* \|^2  - \frac{\alpha}{n}\left(\sum_{i=1}^n f_i(\bar{x}_t) -f_i(x^*)\right)\nonumber\\
    & +\frac{(3L/2+\mu)\alpha}{n} \sum_{i=1}^n \|\bar{x}_t-x_{i,t}\|^2\nonumber\\
    \leq &\left(1-\tfrac{\alpha\mu}{2}\right) \| \bar{x}_t -x^* \|^2  - \frac{\alpha}{n}\left(\sum_{i=1}^n f_i(\bar{x}_t) -f_i(x^*)\right) +\tfrac{5 \alpha L}{2n} \|\Psi_t\|^2,
\end{align}
where the last inequality follows due to $ \|\bar{\textbf{x}}_t-\x_t\|^2\leq \|\Psi_t\|^2$. Let $r_t :=\|\bar{x}_t-x^*\|^2$. Multiplying both sides of \eqref{ll01} by $w_{t+1} = (1-\alpha\mu/4)^{-(t+1)}$, it follows that,
\begin{align}\label{wkineq}
    w_{t+1} r_{t+1} &\leq w_t r_t  -  w_{t+1} \alpha \left( f(\bar{x}_t) -f(x^*)\right) + w_{t+1}\tfrac{5\alpha L}{2n} \|\Psi_t\|^2, 
\end{align}
where $w_{t+1}(1-\alpha\mu/2)\leq w_t$. 

Next, we express \eqref{rbound} (and \eqref{rbound1}) in the form of \eqref{rel} (and \eqref{rel1}) with $a_t =\|\Psi_t\|^2,\,b = 180\alpha^2 L^2 (1+\tap^2)\tap^2,\,c=320n\alpha^4 L^3 (1+\tap^2)\tap$, $e_t = f(\bar{x}_{t}) - f(x^*)$ and $r=0$. By Lemma \ref{lem4}, it follows that,
\begin{equation}\label{rel0}
    \| \Psi_{t}\|^2 \leq 100(1+\tap^2)\left(1- \tfrac{3\rho}{4\tap}\right)^t  \|\Psi_{0}\|^2 + 19200n\alpha^4 L^3(1+\tap^2)\tap^2\sum_{j=0}^{t-1}   \left(1- \tfrac{3\rho}{4\tap}\right)^{t-j}e_j. 
\end{equation}
Note that the condition on the step size \eqref{alpha_bound} ensures that $b< \rho'/4$. Multiplying both sides of \eqref{rel0} by $w_{t+1}:=(1-\alpha\mu/4)^{-(t+1)} $ and summing from $t=0$ to $k-1$ 
\begin{equation}\label{intm101}
\begin{aligned}
    &\sum_{t=0}^{k-1} \left(1- \tfrac{\alpha\mu}{4}\right)^{-(t+1)}   \|\Psi_t\|^2 \\
    \leq &   100(1+\tap^2)\|\Psi_0\|^2\sum_{t=0}^{k-1} \left(1- \tfrac{\alpha\mu}{4}\right)^{-(t+1)} \left(1- \tfrac{3\rho}{4\tap}\right)^t \\
    &+    19200n \alpha^4  L^3(1+\tap^2)\tap^2 \sum_{t=0}^{k-1}\left(1- \tfrac{\alpha\mu}{4}\right)^{-(k+1)} \sum_{j=0}^{t-1} \left(1- \tfrac{3\rho}{4\tap}\right)^{t-j} e_j. 
\end{aligned}
\end{equation}
By \eqref{alpha_bound}, we have 
$\alpha \leq \tfrac{\sqrt{\rho'}}{ L\tap^2}
    \leq \tfrac{1}{2 L\tap^2}
     \leq \tfrac{\rho}{ L\tap}
     \leq  \tfrac{3\rho}{2 \mu\tap}$, 
where the second inequality is due to $\sqrt{\rho'} \leq 1/2$, the third inequality follows by $\tap\geq 1$ and $\rho= 1-2\rho' \geq 1/2$ for $\rho'<1/4$, and the last inequality is due to the fact that $\mu < L$. Thus, it follows that
\begin{equation}\label{a-r}
     \tfrac{\alpha \mu}{2} \leq   \tfrac{3\rho}{4\tap}  \implies    \tfrac{\alpha \mu}{2}(1-\tfrac{\alpha \mu}{8}) \leq   \tfrac{3\rho}{4\tap}\implies 1- \tfrac{3\rho}{4\tap} \leq (1-\tfrac{\alpha \mu}{4})^2.
\end{equation}
We use \eqref{a-r} to bound the two summations on the right-hand-side of \eqref{intm101} as follows
 \begin{equation}\label{fbd1}
    \sum_{t=0}^{k-1} \left(1-\tfrac{\alpha\mu}{4}\right)^{-(t+1)} \left(1- \tfrac{3\rho}{4\tap}\right)^t \leq \sum_{t=0}^{k-1} \left(1-\tfrac{\alpha\mu}{4}\right)^{t-1} \leq \tfrac{4w_1}{\alpha \mu}, 
\end{equation}
and 
\begin{align}\label{fbd2}
     &\sum_{t=0}^{k-1}\left(1-\tfrac{\alpha\mu}{4}\right)^{-(t+1)} \sum_{j=0}^{t-1} \left(1- \tfrac{3\rho}{4\tap}\right)^{t-j} e_j\nonumber  \\ =&\sum_{t=0}^{k-1}\sum_{j=0}^{t-1} \left(1-\tfrac{\alpha\mu}{4}\right)^{-(t+1)+j+1}  \left(1- \tfrac{3\rho}{4\tap}\right)^{t-j} w_{j+1} e_j \nonumber \\
     =& \sum_{t=0}^{k-1}\sum_{j=0}^{t-1}  \left(\tfrac{1- 3\rho/4\tap}{1-\alpha\mu/4}\right)^{t-j} w_{j+1} e_j\nonumber \\
     \leq &
     \sum_{t=0}^{k-1}\sum_{j=0}^{t-1}  \left(1-\tfrac{\alpha\mu}{4}\right)^{t-j} w_{j+1} e_j \nonumber\\
     \leq &\sum_{t=0}^{k-1}  \left(1-\tfrac{\alpha\mu}{4}\right)^{t} \sum_{t=0}^{k-1} w_{t+1} e_t \leq \tfrac{4}{\alpha \mu }\sum_{t=0}^{k-1} w_{t+1} e_t,
\end{align}
where the second inequality is due to \eqref{a-r} and the relation $\sum_{t=0}^{k-1} \sum_{j=0}^{t-1}a_{t-j}b_j $ $\leq\sum_{t=0}^{k-1} a_{t} \sum_{t=0}^{k-1} b_{t}$ for any two non-negative scalar sequences $a_t,\,b_t,\,t\in \mathbb{N}$. By \eqref{fbd1}, \eqref{fbd2} and \eqref{intm101}, it follows that, 
\begin{equation}\label{intm101*}
\begin{aligned}
    &\sum_{t=0}^{k-1} w_{t+1}   \|\Psi_t\|^2 \\
    \leq  &
    \tfrac{400w_1(1+\tap^2)}{\mu\alpha} \|\Psi_0\|^2 + \tfrac{76800 n \alpha^3L^3(1+\tap^2)\tap}{\mu}\sum_{t=0}^{k-1} w_{t+1} \left(f(\bar{x}_t) -f(x^*)\right).
\end{aligned}
\end{equation}
Finally, summing \eqref{wkineq} from $t=0$ to $k-1$, and dividing by $w_t$, it follows that, 
\begin{align*}\label{intm22}
    r_{k} \leq& \tfrac{1}{w_k}\Bigg(w_{0} r_0  + \tfrac{1000w_1(1+\tap^2)L}{n\mu} \|\Psi_0\|^2 \nonumber\\
    & 
    +  \left(\tfrac{ 192000     \alpha^4L^4(1+\tap^2)\tap}{\mu} -\alpha\right)\sum_{t=0}^{k-1} w_{t+1}\left(f(\bar{x}_t) -f(x^*)\right)  \Bigg), 
\end{align*}
where we have used \eqref{intm101*} to bound $\sum_t w_{t+1}\|\Psi_t\|^2$. To prove \eqref{con-res}, we note that $w_k^{-1}=(1-\alpha \mu)^k$ by definition, and the last term in the above inequality is non-positive since $\alpha^3\leq \frac{1}{2\times 307200         L^3\tap^3}$.
\end{proof}

Broadly, Theorem \ref{thm2} establishes the decay of the optimization error for a gradient tracking method with time inhomgeneous weight matrices. The convergence rate of the  algorithm remains linear even when using time-varying matrices, and the form of the convergence factor remains remarkably consistent. However, it is worth mentioning that this convergence factor can potentially be smaller due to the possibility of using smaller step sizes, which depend on the value of $\tap$. In this context, the constant $\tap$  determines the effect of the network on the step size via \eqref{alpha_bound}. More precisely, $\tap$ is a constant chosen to ensure that 
$(1+\tap^2)\left\| \Q[k-\tap :k] - \frac{1}{n} \textbf{1}\textbf{1}^T \right\|^2$ is less than one. This implies that for better connected graphs, i.e., smaller $\left\| \Q[k-\tap :k] - \frac{1}{n} \textbf{1}\textbf{1}^T \right\|^2$, $\tap$ can be smaller so that $\alpha$ can be larger (cf. \eqref{alpha_bound}). For time-inhomogeneous matrices satisfying Assumption~\ref{asmp2}, we can establish precise upper bounds on the value of $\tap$ using the coefficient of ergodicity (cf. Corollary \ref{thm3}).

\begin{remark}\label{gtaremark}
One 
can recover the optimal convergence rate of the \texttt{GTA} algorithm  \cite{DIG}, up to logarithmic factors, from Theorem \ref{thm2}. For \texttt{GTA}, we have $\Q_k=Q\otimes I_d$, for all $k \geq 0$.
Then, for $\tap<k$, if $\tap > \mathcal{O}\left(\frac{1}{\sigma(Q)} \log \frac{1}{\sigma(Q)}\right)$,  
    \begin{align*}
    \rho' = 2(1+\tap^2)\left\| \Q[k-\tap :k] - \frac{1}{n} \textbf{1}_n\textbf{1}_n^T \right\|^2
      &\leq  2(1+\tap^2)\prod_{j=k-\tap}^{k-1}  \left\|\Q- \frac{1}{n} \textbf{1}_n\textbf{1}_n^T\right\|^2 \\       
      &\leq 4\tap^2(1-\sigma(Q))^{2\tap} <1/4.
    \end{align*}
     which implies $\alpha = \tilde{\mathcal{O}}(\frac{\sigma^{2}(Q)}{L})$, where $\tilde{\mathcal{O}}(\cdot)$ hides logarithmic factors. Thus, from Theorem \ref{thm2} we have $\|\bar{x}_T-x^*\|^2\leq \epsilon$, if $T\geq\tilde{\mathcal{O}}\left( \frac{L}{\mu\sigma^{2}(Q)} \log \frac{1}{\epsilon}\right)$.
\end{remark}

We have the following corollary to Theorem \ref{thm2}.
\begin{cor}\label{thm3}
   Suppose that: $(i)$ Assumptions~\ref{asmp1}, \ref{asmp2} and \ref{asmp3} hold, $(ii)$ the matrices $Q_k:= [q_{ij}[k]]_{i\in[n],j\in[n]}$ are doubly stochastic and $\hat{\Q}_k=\Q_k$ for all $k\geq 0$, $(iii)$ $q_{ii}[k]>0$ for all $k\geq 0$ for at least one $i \in [n]$, and, $(iv)$ if $q_{ij}[k]>0$ for any $(i,j) \in \E$ and $k \geq 0$, then $q_{ij}[k]  >q $ for some strictly positive constant $q > 0$ independent of $k$ and $(i,j)$.  Let $\tau_\eta := \eta \ta d_\G$, where $\ta$ is defined in Assumption \ref{asmp2} and $\eta \in \mathbb{N}$ satisfies 
\begin{equation}\label{c-cond}
    \eta \geq \left\lceil \tfrac{\max\{\ln 16n^3\ta^2d^2_\G,16 \ln 4/\gamma \}}{\gamma} \right\rceil
\end{equation}
where 
$\gamma:=q^{  d_\G \ta}$. Then, if $\alpha =\mathcal{O}\left( \tfrac{1}{L\tau_\eta^{2}}\right)$, (\ref{con-res}) is satisfied for $\bar{x}_k$ generated via the recursions (\ref{x-iter})-(\ref{y-iter}).
\end{cor}
\begin{proof}
To prove the corollary, we need to show that there exists a constant $ \eta \in \mathbb{N}$ such that for $\tau_\eta =\eta\ta d_\G$, we have  $\rho' := 2(1+\tau_\eta^2) \left\| \Q[j-\tau_\eta :j] - \tfrac{1}{n} \textbf{1}_n\textbf{1}_n^T \right\|<1/4$ for any 
$\tau_\eta\leq j\leq t$ to ensure the results of Theorem \ref{thm2} hold with $\tap=\tau_\eta$. It follows that
\begin{align*}
    \delta(\Q[(j-\tau_\eta:j]) &\leq  \rho\left( \Q[(j-\tau_\eta:j]\right) \nonumber\\
    &\leq \rho(   \mathbf{Q}[j- \eta \ta d_\G:j-(\eta-1)\ta d_\G])\cdots \rho(\mathbf{Q}[(j-\ta d_\G:j])) \nonumber\\
    &\leq \left( 1- \gamma\right)^{\eta},
\end{align*}
where $\gamma:= q^{\ta d_\G}$, and the first, second and third inequalities are due to \eqref{upbound}, \eqref{cross} and \eqref{upbound0}, respectively.  
Following the same logic as in \eqref{OI}, it follows that,
\begin{align}\label{tmpI}
    \left\| \Q [j-\tau_\eta:j]  - \tfrac{1}{n} \textbf{1}_n\textbf{1}_n^T \right\|_1 &\leq  n \delta(\Q[j-\tau_\eta:j] ) 
    \leq  n\left( 1- \gamma\right)^{\eta} 
    \leq n\exp(-\gamma \eta).
\end{align} 
Consequently, this implies,
\begin{align}\label{tmp6}
    2(1+\tau_\eta^2)\left\| \Q[j-\tau_\eta:j]  - \tfrac{1}{n} \textbf{1}_n\textbf{1}_n^T\right\|^2 &\leq 2(1+\tau_\eta^2)n\left\| \Q[j-\tau_\eta:j]  - \tfrac{1}{n} \textbf{1}_n\textbf{1}_n^T\right\|^2_{1}\nonumber\\
    &\leq 2n^3 (1+\eta^2\ta^2 d^2_\G)\exp(-2\gamma \eta)\nonumber\\
    &\leq \underbrace{4 n^3 \ta^2 d^2_\G}_{:=A} \eta^2 \exp(-2\gamma \eta),
\end{align}
where the second inequality is due to \eqref{tmpI} and the last inequality follows since $\eta\ta d_\G\geq 1 $. 
We next prove the following claim for any scalars $\eta,A\geq 1$ and $0<\gamma<1$:
\begin{align}\label{claim}
    \eta^2\exp(-2\gamma \eta) < \tfrac{1}{4A} \qquad \text{if}\qquad \eta > \left\lceil\max\left\{ \tfrac{\ln 4A, 16\ln 4/\gamma}{\gamma}\right\}\right\rceil .
\end{align}
To prove the claim, we note that the assumed inequality implies $\left( 1 -\tfrac{\ln \eta }{ \gamma \eta }\right)\eta > \tfrac{\ln 4A}{2\gamma}$. 
Let $\tilde{\eta}\in \mathbb{R}$ be such that $0< \ln \tilde{\eta}/\gamma \tilde{\eta} < 1/4$. Then, for any $\eta > \tilde{\eta}$, 
\begin{equation}\label{kadash}
    \eta \geq  \tfrac{2\ln 4A}{3\gamma}.
\end{equation}
To prove the existence of a $\tilde{\eta}$ satisfying $\ln \tilde{\eta}/ \tilde{\eta} \leq \gamma/4:=\epsilon$, we consider $\tilde{\eta} = \frac{4\ln 1/\ep}{\ep},\,\ep<\frac{1}{4}$. For such a $\tilde{\eta}$, we have, $\ln \tilde{\eta}/\tilde{\eta}= \epsilon \tfrac{\ln \frac{4}{\epsilon}+ \ln \ln \frac{1}{\ep}}{4\ln 1/\epsilon}<\epsilon$. 
Combining \eqref{kadash} with $A= 4n^3\ta^2d_\G^2$ and $\eta\geq \tilde{\eta} = 16\ln  \big(4/\gamma \big)/ \gamma$ gives the lower bound on $\eta$ in \eqref{claim}. Finally, by \eqref{claim}, \eqref{tmp6} can be bounded as,  
 $2(1+\tau_\eta^2)\left\| \Q[j-\tau_\eta:j]  - \tfrac{1}{n} \textbf{1}_n\textbf{1}_n^T\right\|^2  \leq A \eta^2 \exp(-2\gamma \eta)<\tfrac{1}{4}$,
which completes the proof.
\end{proof}

The exact convergence rate of \texttt{AC-GT} can be derived from Corollary \ref{thm3}. By \eqref{con-res}, the number of iterations required to reach $\epsilon$-accuracy, denoted by $T$, is of the order of $\mathcal{O}\left(\frac{L\tau_\eta^2}{\mu}  \log \frac{1}{\epsilon}\right)$  since  $\alpha =\mathcal{O}\left( \frac{1}{L\tau_\eta^{2}}\right)$. Using \eqref{c-cond} to bound $\eta$ in $\tau_\eta = \eta\ta d_\G$, it follows
\begin{align*}
    T=\mathcal{O}\left( \tfrac{L \eta^2 \bar{\tau}^2d^2_\G}{\mu } \log \tfrac{1}{\epsilon} \right)= \tilde{\mathcal{O}}\left( \left(\tfrac{\bar{\tau}^2d^2_\G}{\gamma^2}\right) \tfrac{L}{\mu   }  \log \tfrac{1}{\epsilon}\right).
\end{align*}
where $\tilde{\mathcal{O}}(\cdot)$ hides logarithmic factors. Compared to the iteration complexity of \texttt{GTA} (see Remark \ref{gtaremark}) under the connected graph assumption, we note that the number of iterations can potentially increase by a factor of $\tilde{\mathcal{O}}\left( \frac{\bar{\tau}^2d^2_\G}{\gamma^2}\right)$. This is expected given the  weaker assumptions made, i.e., not requiring the graph to be connected at every iteration (Assumption \ref{asmp2}). Despite the increased iteration complexity, one can potentially have savings in overall communication volume for \texttt{AC-GT} (cf. Section \ref{acgt_exp}) analogous to those for \texttt{AC} (cf. Remark \ref{remnew}). 

\section{Numerical Experiments}\label{sec.num_red}
In this section, we illustrate the empirical performance of \texttt{AC} and \texttt{AC-GT} 
via two sets of experiments. The first set of experiments demonstrates the benefits of \texttt{AC} 
compared to the 
distributed averaging algorithm in achieving consensus and illustrates the effect of the parameters of the  pruning protocol on the performance of \texttt{AC}. The second set of experiments show the merits of \texttt{AC-GT} 
compared to popular methods on a linear regression problem with synthetic data 
\cite{synthetic}, and a logistic regression problem with real datasets \cite{mushroom, aust} from the UCI repository \cite{UCI}. All methods are implemented in Python, with a dedicated CPU core functioning as a node.

\subsection{Performance of \texttt{AC}}\label{sec4.1}
We first showcase the effectiveness of \texttt{AC} in achieving consensus, where the goal is for all nodes to attain the average value of the initial estimates of the nodes \cite[Section 1]{randomgossip}. The network topologies (graphs) are generated randomly using the Erd\"{o}s-R\'{e}nyi graph model \cite{renyi} and are represented as $G(n,p)$, where $n$ represents the number of nodes, and $p \in \{0.2,0.4,0.6,0.8\}$ denotes the probability with which each possible edge is independently included in the pruned graph. The performance metric used is the average consensus error, defined as $\frac{1}{|\E|}\sum_{(i,j)\in \E} \|x_i-x_j\|$, where $\E$ represents the set of all edges and $x_i \in \mathbb{R}^d$ for all $i \in [n]$ 
with $d=10$. The total communication volume is measured as the total number of vectors exchanged amongst all the nodes in the network. The initial values $\{x_{i,0}\}_{i\in [n]}$ 
at each node are generated following a standard normal distribution.

\begin{figure}
     \centering
     \begin{subfigure}{0.3\textwidth}
         \centering
         \includegraphics[width=\textwidth]{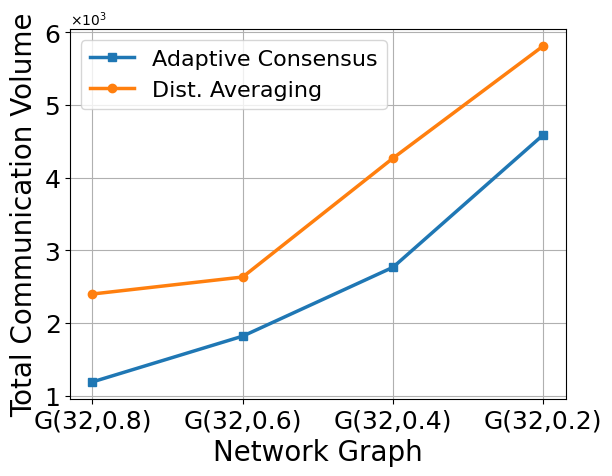}
         \caption*{\tiny{\textbf{(a)}}}
     \end{subfigure}
     \begin{subfigure}{0.3\textwidth}
         \centering
         \includegraphics[width=\textwidth]{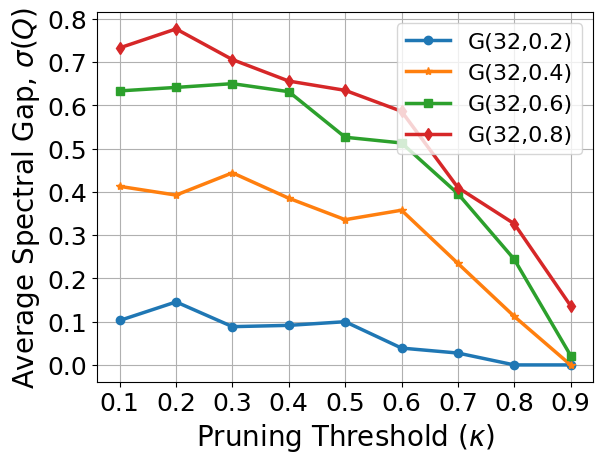}
         \caption*{\tiny{\textbf{(c)}}}
     \end{subfigure}
     \begin{subfigure}{0.3\textwidth}
         \centering
         \includegraphics[width=\textwidth]{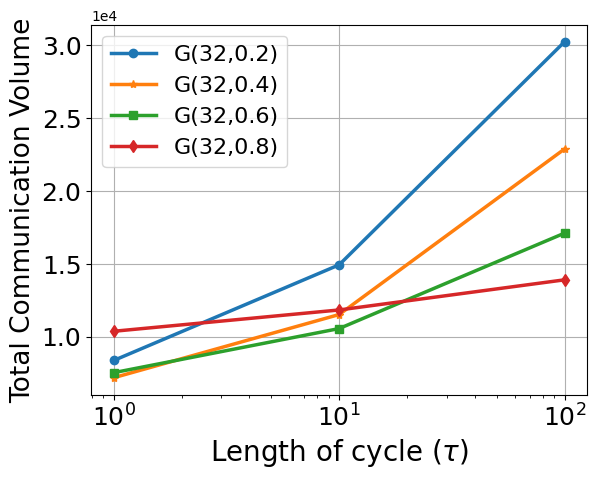}
         \caption*{\tiny{\textbf{(e)}}}
     \end{subfigure}
      \centering
     \begin{subfigure}{0.3\textwidth}
         \centering
         
         \includegraphics[width=\textwidth]{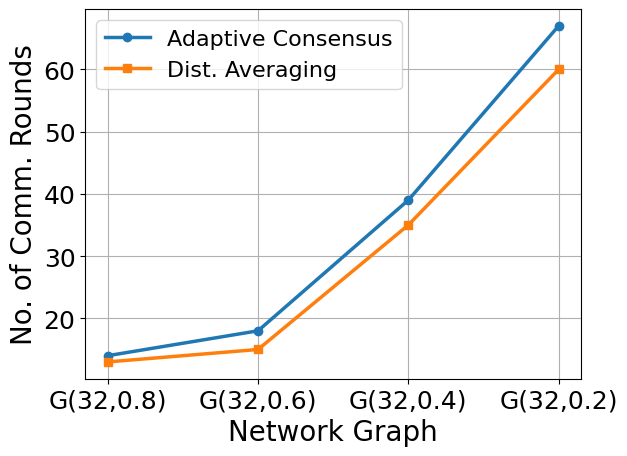}
         \caption*{\tiny{\textbf{(b)}}}
     \end{subfigure}
     \begin{subfigure}{0.3\textwidth}
         \centering
         \includegraphics[width=\textwidth]{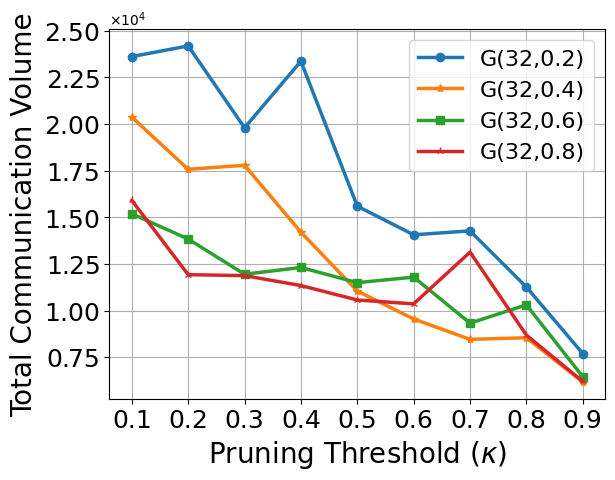}
         \caption*{\tiny{\textbf{(d)}}}
     \end{subfigure}
     \begin{subfigure}{0.3\textwidth}
         \centering
         \includegraphics[width=\textwidth]{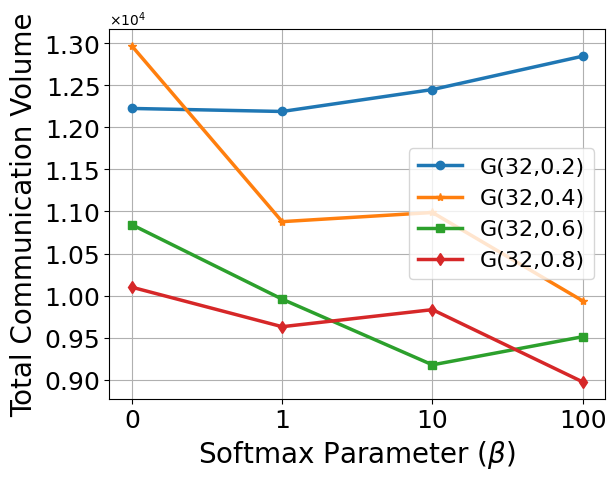}
         \caption*{\tiny{\textbf{(f)}}}
     \end{subfigure}
         
         \vspace{-0.5cm}\caption{\small{\textbf{(a)-(b)}  Total communication \textit{volume}/\textit{rounds} required to achieve a 
         consensus error of 
         $10^{-10}$. \textbf{(c)} Variation of spectral gap with respect to pruning threshold, $\kappa \in \{ 0.1,0.2,\dots,0.9\}$. \textbf{(d)-(f)} Total communication volume required to achieve a 
         consensus error of 
         $10^{-10}$ for different $\kappa \in \{ 0.1,0.2,\dots,0.9\}$, $\tau \in \{1,10^1,10^2\}$ and $\beta \in \{0,1,10^1,10^2\}$, respectively.
         }}
    \label{fg1}    
\end{figure}

\paragraph{Comparison to distributed averaging}
Figs. \ref{fg1}\textbf{(a)-(b)} compare the performance of \texttt{AC} to distributed averaging \cite{rensurvey}. The latter can be considered a specific case of \texttt{AC} with $\bar{\kappa}=0$ and  $\tau=\infty$. For the pruning protocol part of \texttt{AC}, we have set $\bar{\kappa}_i=\kappa=0.75$ for all $i \in [n]$ 
and choose 
$\ubar{\kappa}_i$ to ensure that $|\E^i|\geq 1$, so that each node has at least one neighbor. The softmax parameter is set to $\beta=1$ and the cycle length is set to $\tau=10$. The mixing matrix is generated using the Metropolis Hastings rule (cf. \eqref{MH}).
Fig. \ref{fg1}\textbf{(a)} shows a significant reduction in the total communication volume required to reach a consensus error of $10^{-10}$ as compared to distributed averaging across all graph topologies. Fig. \ref{fg1}\textbf{(b)} demonstrates that the number of communication rounds for \texttt{AC} undergoes 
only a modest increase as compared to distributed averaging.

\paragraph{Variation of pruning threshold ($\kappa$)}
In Fig. \ref{fg1}\textbf{(c)}, we plot the average spectral gap of the mixing matrices as a function of $\kappa \in \{ 0.1,0.2,\dots,0.9\}$. The average spectral gap is defined as the average of the spectral gaps of all the weight matrices obtained throughout the pruning cycles in a run of the algorithm. The plot reveals an important observation: pruning up to 50-60\% of the edges does not significantly affect the spectral properties of the mixing matrix. Moreover, increasing the value of $\kappa$ leads to a decrease in communication volume across all graphs, see 
Fig. \ref{fg1}\textbf{(d)}. 


\paragraph{Variation of consensus cycle length ($\tau$)}
Intuitively, one expects \texttt{AC} to perform better with shorter cycles 
since more frequent pruning of the graph can potentially allow \texttt{AC} to adapt more effectively to varying consensus errors. 
Fig. \ref{fg1}\textbf{(e)} confirms this intuition, 
where we consider $\tau \in  \{1, 10, 100\}$ with $\kappa = 0.75$. While a value of $\tau=1$ yields optimal performance in terms of communication volume, it necessitates executing the pruning protocol at every iteration.

\paragraph{Variation of softmax parameter ($\beta$)}
Fig. \ref{fg1}\textbf{(f)} plots the total communication volume required to achieve 
a consensus error of $10^{-10}$ as a function of the softmax parameter $\beta \in \{0,1,10,100\}$ with $\tau=10$ and $\kappa=0.75$. The total communication volume is obtained by averaging over 100 independent trials. Fig. \ref{fg1}\textbf{(f)} shows that  higher values of $\beta$ tend to show a modest improvement in the performance.

\subsection{Performance of \texttt{AC-GT}}\label{acgt_exp}
This subsection considers the evaluation of the performance of \texttt{AC-GT} on linear and logistic regression problems.

\subsubsection{Linear Regression}\label{acgt_exp_1}
We first consider a linear least-squares regression problem with synthetic data, formally defined as,
\begin{align*}
    \min_{x\in \mathbb{R}^d} \;\; f(x):=\tfrac{1}{N}\sum_{i=1}^N (a^T_ix-b_i )^2,
\end{align*}
where $a_i \in \mathbb{R}^{d}$ denotes the $i$th feature vector and $b_i\in \mathbb{R}$ denotes the corresponding label. The data is generated using the technique proposed in \cite{synthetic} with $N=32000$ and $d = 10$. 
The network topologies considered are $G(n,p)$, where $n=32$ and $p \in \{0.2,0.5,0.8\}$. The data is partitioned uniformly in a disjoint manner amongst the nodes. We tuned the step size parameter in \texttt{AC-GT} using a grid-search over the range $\alpha \in \{10^{-4},10^{-3},10^{-2},10^{-1} , 10^{0}\}$ and present the results for the best step size. The softmax parameter is  set to $\beta=1$ and the cycle length is set to $\tau=10$. The mixing matrix is generated using the Metropolis Hastings rule (cf. \eqref{MH}). 

\begin{figure}
     \centering
     \begin{subfigure}{0.3\textwidth}
         \centering
         \includegraphics[width=\textwidth]{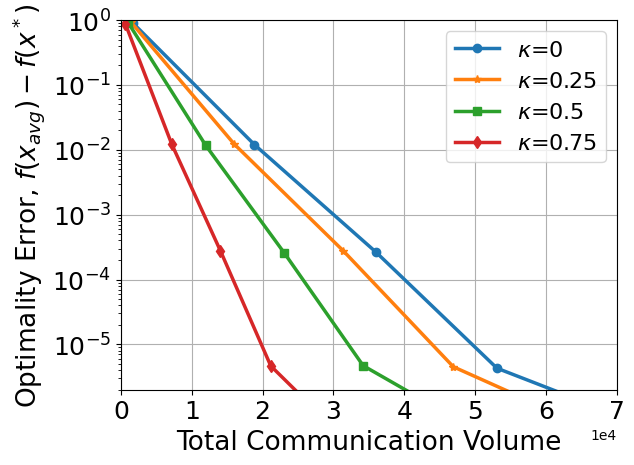}
     \end{subfigure}
     \begin{subfigure}{0.3\textwidth}
         \centering
         \includegraphics[width=\textwidth]{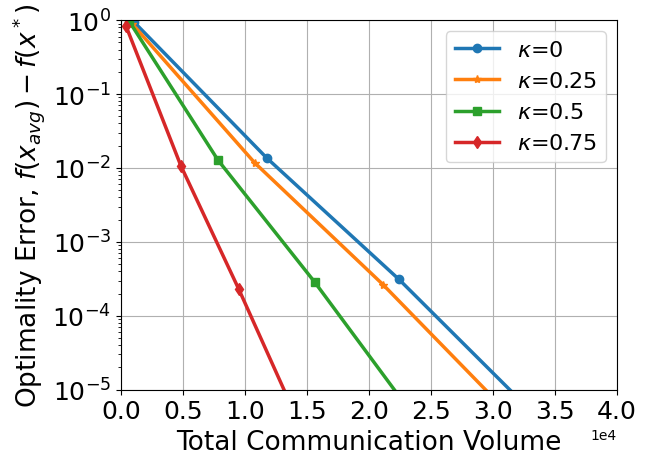}
     \end{subfigure}
     \begin{subfigure}{0.3\textwidth}
         \centering
         \includegraphics[width=\textwidth]{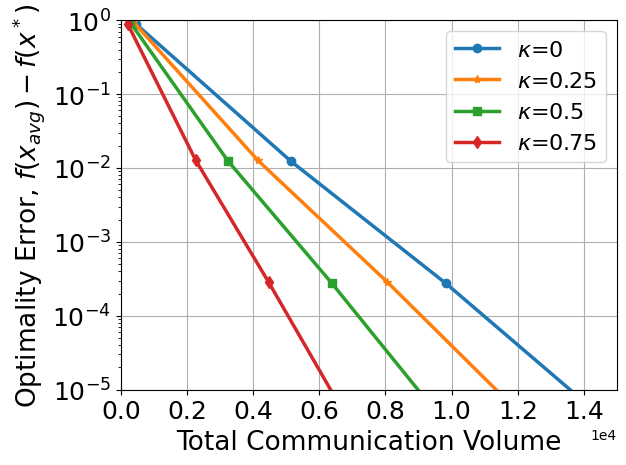}
     \end{subfigure}
     \begin{subfigure}{0.3\textwidth}
         \centering\includegraphics[width=\textwidth]{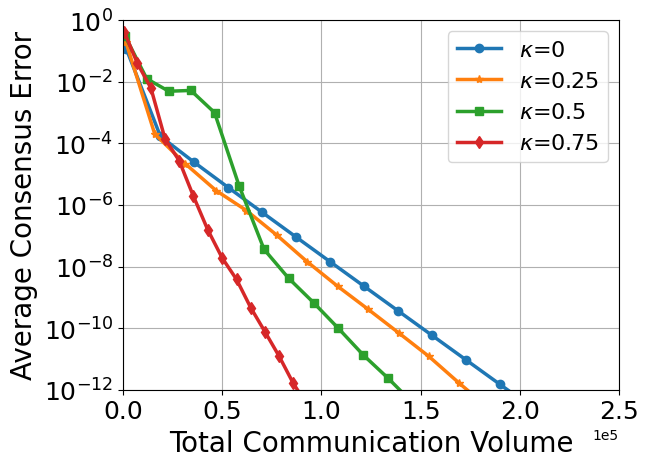}
          \caption*{\textbf{\tiny{(a)  G(32,0.8)}}}
     \end{subfigure}
     \begin{subfigure}{0.3\textwidth}
         \includegraphics[width=\textwidth]{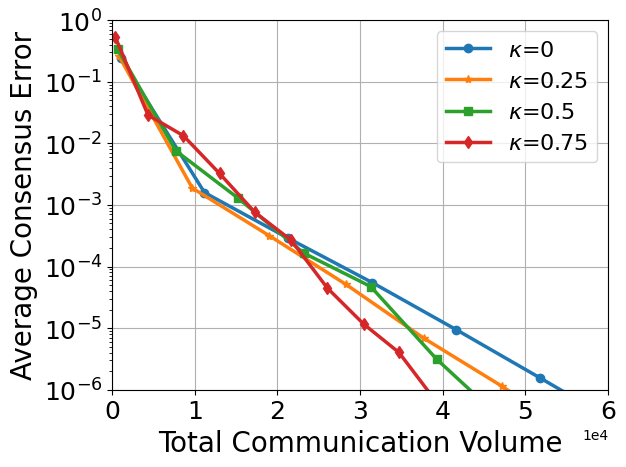}
         \caption*{\textbf{\tiny{(b)  G(32,0.5)}}}
     \end{subfigure}
     \begin{subfigure}{0.3\textwidth}
         \includegraphics[width=\textwidth]{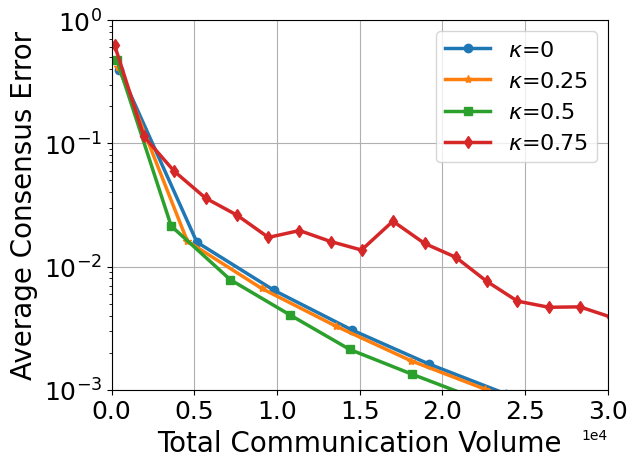}
         \caption*{\textbf{\tiny{(c)  G(32,0.2)}}}
     \end{subfigure}
         
         \vspace{-0.5cm}\caption{\small Performance of \texttt{AC-GT} on linear regression problems for three different graphs, \textbf{(a)} G(32,0.8) \textbf{(b)} G(32,0.5) \textbf{(c)} G(32,0.2). \textbf{Top}: Optimality Error versus Total Communication Volume.  \textbf{Bottom}: Average Consensus Error versus Total Communication Volume.}
    \label{fig4}
\end{figure}

Fig. \ref{fig4} illustrates the performance of \texttt{AC-GT} in terms of two metrics, optimality error, defined as $f(x_{\text{avg}}) - f(x^*)$, where $x_{\text{avg}}=\frac{1}{n}\sum_{i=1}^n x_i$, and  average consensus error described in Section \ref{sec4.1}, with respect to the total communication volume. The results suggest that, in terms of optimality error, it is preferable to use a higher value of $\kappa$, the pruning threshold. This observation is consistent across graph topologies. That said, there is a slight degradation in the decay of the consensus error as $\kappa$ increases. This degradation becomes more noticeable in sparser topologies, as seen in Fig. \ref{fig4}\textbf{(c)}.

\subsubsection{Logistic Regression}
We consider $\ell_2$-regularized logistic regression problems with real datasets of the form,
\begin{align*}
    \min_{x\in \mathbb{R}^d} \;\; f(x):=-\tfrac{1}{N}\sum_{i=1}^N \left\{ b_i\log \sigma(a_i^Tx) + (1-b_i)\log\left( 1-\sigma(a_i^Tx) \right)\right\} + \tfrac{\lambda}{2} \|x\|^2
\end{align*}
where $\{a_i,b_i\}_{i=1}^N$ represent the training samples with label $b_i\in \{0,1\}$, $\lambda > 0$ is the regularization parameter and $\sigma(z) = \tfrac{1}{1+\exp(-z)},$ $\forall z\in \mathbb{R}$ is 
the sigmoid function.

We consider the Statlog \cite{aust} and the Mushroom \cite{mushroom} datasets from the UCI repository \cite{UCI}. The Statlog dataset consists of $N=690$ samples and $d=14$ features whereas the Mushroom dataset consists of $N=8124$ samples and $d=22$ features. For these experiments, we consider $G(n,p)$ with $n=16$ and $p=0.5$.  The data partition and the algorithm parameters for \texttt{AC-GT} are set in the same manner as Section \ref{acgt_exp_1}. The step size is tuned using a grid-search over the range $\alpha \in  \{10^{-4},10^{-3},10^{-3},10^{-1},10^{0}\}$ for all the algorithms. The regularization parameter is set to $\lambda =10^{-4}$. The optimal solution $x^*$ is computed using the L-BFGS algorithm from the SciPy library in Python and solving the problems to high accuracy.

\begin{figure}
     \centering
     \begin{subfigure}{0.3\textwidth}
         \centering
         \includegraphics[width=\textwidth]{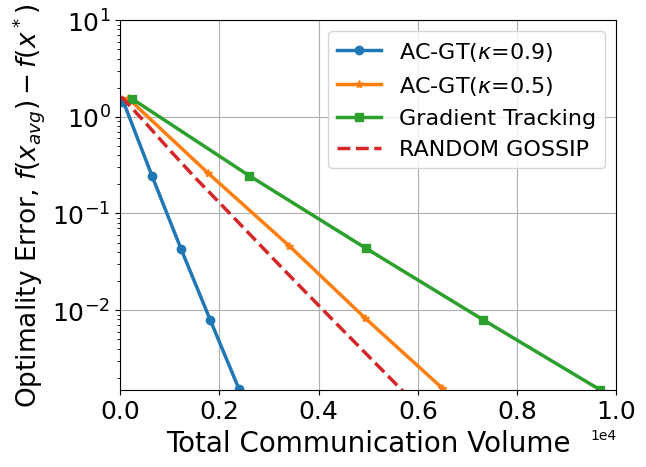}
     \end{subfigure}
     \begin{subfigure}{0.3\textwidth}
         \centering
         \includegraphics[width=\textwidth]{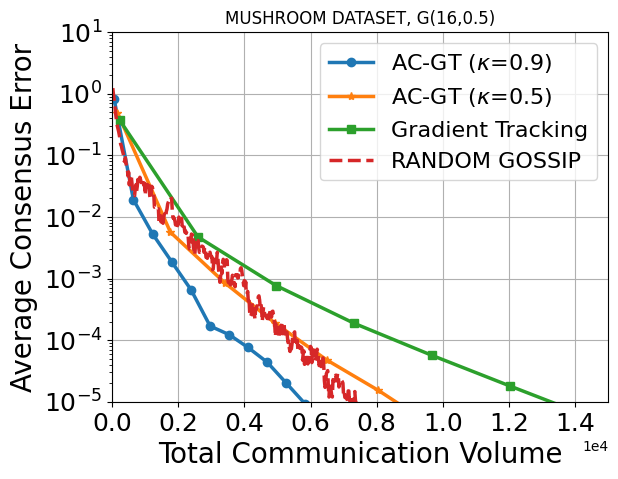}
     \end{subfigure}
     \begin{subfigure}{0.3\textwidth}
         \centering
         \includegraphics[width=\textwidth]{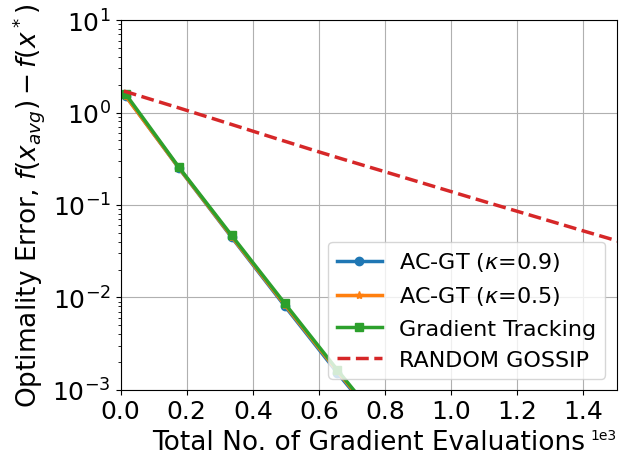}
     \end{subfigure}
      \centering
     \begin{subfigure}{0.3\textwidth}
         \centering
         \includegraphics[width=\textwidth]{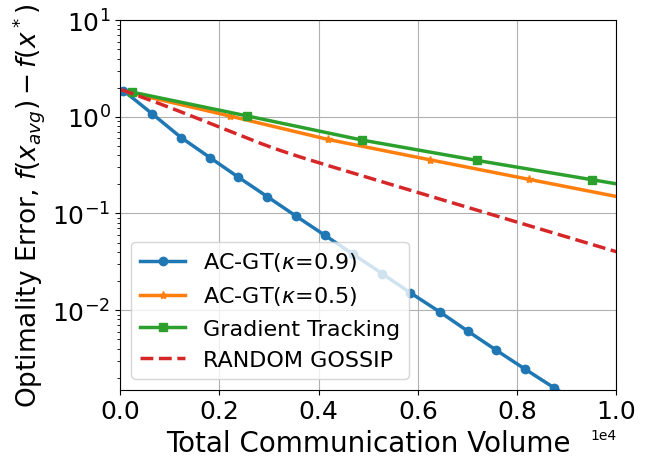}
          \caption*{\tiny{(a)}}
     \end{subfigure}
     \begin{subfigure}{0.3\textwidth}
         \centering
         \includegraphics[width=\textwidth]{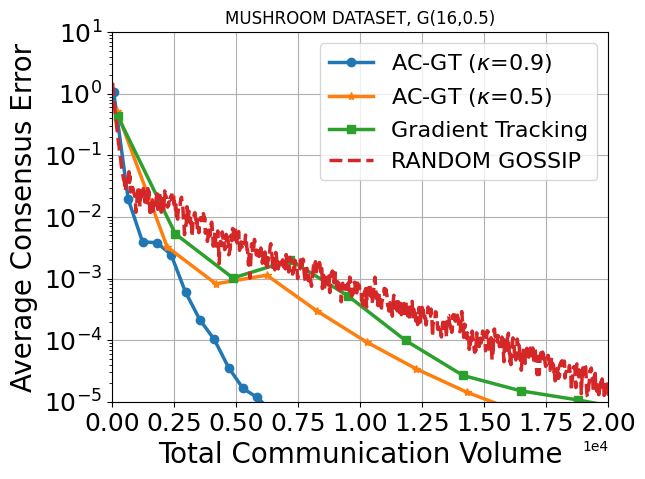}
         \caption*{\tiny{(b)}}
     \end{subfigure}
     \begin{subfigure}{0.3\textwidth}
         \centering
         \includegraphics[width=\textwidth]{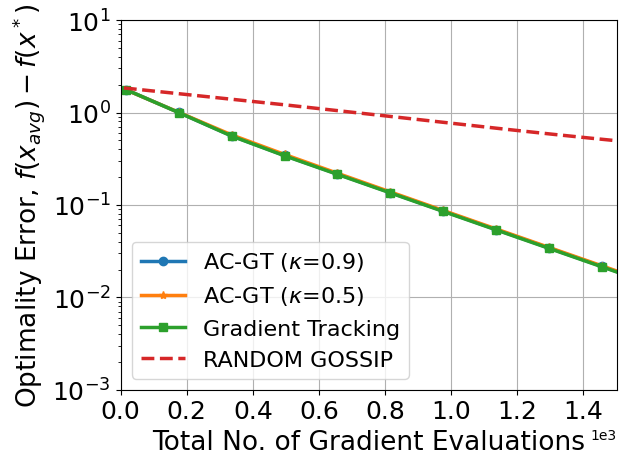}
         \caption*{\tiny{(c)}}
     \end{subfigure}
           \vspace{-0.5cm}\caption{Performance of \texttt{AC-GT} on logistic regression problems: \textbf{(a)} Optimality Error versus Total Communication Volume \textbf{(b)} Consensus Error versus Total Communication Volume \textbf{(c)} Optimality Error versus Total number of Gradient Evaluations. \textbf{Top:} Statlog Dataset, G(16,0.5). \textbf{Bottom:} Mushroom Dataset, G(16,0.5).}
    \label{fig6}
\end{figure}

The performance of \texttt{AC-GT} is compared to EXTRA  \cite{EXTRA} a popular gradient tracking algorithm (denoted by ``Gradient Tracking'' in the plots) and the random gossip algorithm \cite{randomgossip}\footnote{To solve the semi-definite problem required for implementing the random gossip algorithm from \cite{randomgossip}, we utilize the CVXPY library \cite{cvxpy}.}. In addition to the previous metrics, we also report the optimality error versus the total number of gradient evaluations of $f(\cdot)$. From the optimality error plots shown in Figs. \ref{fig6}\textbf{(a)} and \textbf{(c)}, it is evident that \texttt{AC-GT} with a parameter value of $\kappa=0.9$ exhibits the best performance. While the optimality error of random gossip is comparable to \texttt{AC-GT} with $\kappa=0.5$ in terms of total communication volume, \texttt{AC-GT} outperforms the former with respect to total gradient evaluations. As for the consensus error, there is no notable difference in algorithm performance for the Statlog dataset. However, for the Mushroom dataset, random gossip and gradient tracking appear to exhibit inferior performance. 

\section{Conclusion}\label{sec.conc}

In this paper, we have developed an adaptive randomized algorithmic framework aimed at enhancing the communication efficiency of decentralized algorithms. Based on this framework, we have proposed the \texttt{AC} algorithm to solve the consensus problem and the \texttt{AC-GT} algorithm to solve the decentralized optimization problem. The distinguishing feature of the framework is the ability to reduce the volume of communication by making use of the inherent network structure and local information. We have established theoretical convergence guarantees and have analyzed the impact of various algorithmic parameters on the performance of the algorithms. Numerical results on the consensus problem, and linear and logisitc regression problems, demonstrate that proposed algorithms achieve significant communication savings as compared to existing methodologies.

Finally, several interesting extensions of the proposed algorithmic framework can be considered. From a communication perspective, one could consider directed graphs. Most of the groundwork for this setting has already been laid out in this work and as mentioned earlier, the theory can be extended to accommodate push-pull gradient methods \cite{pupush}, where either row or column stochasticity is satisfied. Additionally, asynchronous updating within each consensus cycle can also be incorporated to alleviate the constraints imposed by slower (straggler) nodes. Other interesting directions include nonconvex problems,  stochastic local information and inexact communication.

\newpage
\bibliographystyle{plain}
{\small
\bibliography{ref}}


\end{document}